\documentclass{amsart}

\usepackage{amsmath}
\usepackage{mathrsfs}   
\usepackage{amsthm}  
\usepackage[linktocpage]{hyperref}
\usepackage{amsfonts,graphics,amsthm,amsfonts,amscd,latexsym}
\usepackage{epsfig}
\usepackage{enumitem}
\usepackage{flafter}
\usepackage{mathtools}
\usepackage{comment}
\usepackage{stmaryrd}
\usepackage{adjustbox}
\usepackage{stmaryrd}
\usepackage{enumitem}
\usepackage{makecell}
\usepackage[title]{appendix}
\hypersetup{
    colorlinks=true,    
    linkcolor=red,          
    citecolor=green,      
    filecolor=blue,      
    urlcolor=blue           
}
\usepackage{tikz}
\usetikzlibrary{graphs,positioning,arrows,shapes.misc,decorations.pathmorphing}

\tikzset{
    >=stealth,
    every picture/.style={thick},
    graphs/every graph/.style={empty nodes},
}

\tikzstyle{vertex}=[
    draw,
    circle,
    fill=black,
    inner sep=1pt,
    minimum width=5pt,
]
\usepackage[position=top]{subfig}
\usepackage[alphabetic]{amsrefs}
\usepackage{amssymb}
\usepackage{color}
\usepackage{xcolor}

\calclayout

\usetikzlibrary{decorations.pathmorphing}
\tikzstyle{printersafe}=[decoration={snake,amplitude=0pt}]

\newcommand{\Pic}{\operatorname{Pic}}

\newcommand{\supp}{\operatorname{supp}}

\newcommand{\ee}{\mathcal{E}}

\newcommand{\oo}{\mathcal{O}}

\def\O#1.{\mathcal {O}_{#1}}			
\def\pr #1.{\mathbb P^{#1}}				
\def\af #1.{\mathbb A^{#1}}			
\def\ses#1.#2.#3.{0\to #1\to #2\to #3 \to 0}	
\def\xrar#1.{\xrightarrow{#1}}			
\def\K#1.{K_{#1}}						
\def\bA#1.{\mathbf{A}_{#1}}			
\def\bM#1.{\mathbf{M}_{#1}}				
\def\bL#1.{\mathbf{L}_{#1}}				
\def\bB#1.{\mathbf{B}_{#1}}				
\def\bK#1.{\mathbf{K}_{#1}}			
\def\subs#1.{_{#1}}					
\def\sups#1.{^{#1}}

\usepackage{tikz}
\usetikzlibrary{matrix,arrows,decorations.pathmorphing}

  \newtheorem{theorem}{Theorem}[section]
  \newtheorem{lemma}[theorem]{Lemma}
  \newtheorem{proposition}[theorem]{Proposition}
  \newtheorem{corollary}[theorem]{Corollary}

  \newtheorem{definition}[theorem]{Definition}

\newtheorem{remark}[theorem]{Remark}

\theoremstyle{remark}

\numberwithin{equation}{section}

\usepackage[all]{xy}

\newcounter{rownumber}[figure] 
\newcounter{rownumber-irr}[figure] 
\newcounter{rownumber-p1}[figure] 
\begin{document}

\title{K-polystability of $3$-dimensional log Fano pairs of~Maeda type}

\author[K.~Loginov]{Konstantin Loginov}
\email{loginov@mi-ras.ru}
\address{Steklov Mathematical Institute of Russian Academy of Sciences, Moscow, Russia.}


\begin{abstract}
Using the Abban-Zhuang theory and the classification of three-dimensional log smooth log Fano pairs due to Maeda, we prove that almost all threefold log Fano pairs $(X, D)$ of Maeda type with reducible boundary $D$ are K-unstable for any values of the boundary coefficients. We also correct several inaccuracies in Maeda's classification.
\end{abstract}
\maketitle


\tableofcontents
\section{Introduction}

K-stability is an algebraic invariant that characterises the existence of K\"ahler-Einstein metrics on Fano varieties over the field of complex numbers. More precisely, by \cite{CDS15, Tia15}, a K\"ahler-Einstein metrics exists on a smooth Fano variety if and only if the variety is K-polystable. For a survey on K-stability, see \cite{Xu21}. In \cite{ACCFKMSSV21}, K-stability of smooth three-dimensional Fano varieties was investigated, and the problem of characterising K-polystable varieties was solved for a general element in each of $105$ families of smooth Fano threefolds. 

In this paper, we treat a similar problem for three-dimensional log Fano pairs. K-stability for pairs was defined in \cite{Don12}, see also \cite{OS15}. K-stability for pairs, also called log K-stability, is connected with the existence of K\"ahler-Einstein edge (KEE) metrics, see \cite{R14} for a survey. 

Following \cite{Fu20a}, we work with \emph{log Fano pairs of Maeda type}, that is, klt log Fano pairs $(X, D)$ with $D\neq 0$ such that there is a log smooth log Fano pair $(X, \Delta)$ satisfying $\mathrm{Supp}(D)\subseteq\mathrm{Supp}(\Delta)$, where $\Delta$ is an integral divisor. Such pairs provide natural examples to the study of K-stability of pairs. In contrast to the situation with smooth Fano varieties, there are infinitely many pairwise non-isomorphic varieties $X$ such that $(X, \Delta)$ is a log smooth log Fano pair of fixed dimension~$n\geq 2$.

In dimension two, in \cite{Fu20a} it was shown that two-dimensional log Fano pairs of Maeda type are K-unstable, with two exceptions, see Theorem~\ref{thm-2-dim}. In arbitrary dimension, by \cite{Fu16b} it is known that, for a log Fano pair of Maeda type, if the coefficients of the boundary are sufficiently close to $1$, then this pair is K-unstable, see Theorem \ref{thm-G-small-values}. This answers a conjecture proposed in \cite{CR15}, where a more general notion of asymptotically log Fano varieties was introduced. In \cite{CR18}, it was proven that a pair which belongs to the class of asymptotically log del Pezzo surfaces is K-unstable, if the values of the boundary coefficients are close to $1$.  

We prove that in dimension $3$ and in the case of reducible $D$, log Fano pairs of Maeda type $(X, D=\sum c_i D_i)$ are K-unstable for any values of the coefficients $0\leq c_i<1$, with five exceptions. More precisely, we prove the following theorem.

\begin{theorem}
\label{main-thm}
Let $(X, D)$ be a log Fano pair of Maeda type with $\dim X = 3$. Assume that $D$ is reducible. Then the pair $(X, D)$ is K-semistable if and only if it belongs to the following list:
\begin{enumerate}
\item
\label{main-thm-case1}
$(\mathbb{P}^3, aD_1 + bD_2)$ where $D_1$ is a smooth quadric, $D_2$ is a plane, $3b\leq2a$, and $6a-b\leq 4$.
\item
\label{main-thm-case2}
$(Q, aD_1 + bD_2)$ where $Q\subset\mathbb{P}^4$ is a smooth quadric, $D_i$ are hyperplane sections, $3a-b\leq1$, and $3b-a\leq 1$.
\item
\label{main-thm-case3}
$(\mathbb{P}^1\times \mathbb{P}^2, aD_1 + bD_2)$ where $D_1\sim H+F$, $D_2\sim F$, $H$ is pullback of a line on $\mathbb{P}^2$, $F$ is a pullback of a point on $\mathbb{P}^1$, and one of the following conditions holds:

$0\leq a\leq \alpha\approx0.507$, and $b\leq a/2$, where $\alpha$ is a root of the equation $-40+106\alpha-58\alpha^2+9\alpha^3=~0$, 

or $\alpha\approx0.507\leq a\leq 20/37$, and  

$(-4+13 a-4 a^2)/(3 a)-\sqrt{(32-88 a+84 a^2-34 a^3+5 a^4)/a^2}/(3 \sqrt{2})\leq b\leq a/2$.
\item
\label{main-thm-case4}
$(X, aD_1 + bD_2)$, where $X$ is the blow up 
of quadric $Q\subset \mathbb{P}^4$ in a plane conic $C\subset Q$, $D_1=E$ is the exceptional divisor of the blow up $X\to Q$, $D_2$ is a preimage of a hyperplane section of $Q$ that contains the conic $C$, and the following conditions hold: 

$0\leq a\leq(\sqrt{10}-2)/3$ and $0\leq b\leq 1-\sqrt{4+4a+3a^2}/\sqrt{6}$.
\item
\label{main-thm-case5}
$(X, aD_1 + bD_2)$, where $X$ is a divisor of bidegree $(1, 1)$ in $\mathbb{P}^2\times\mathbb{P}^2$, $D_1$ is a divisor of bidegree $(1, 0)$, $D_2$ is a divisor of bidegree $(0, 1)$, and 

$0\leq a\leq 2-\sqrt{3}$ and 
$0\leq b\leq (16-3 a)/8 - \sqrt{3} \sqrt{64-32 a+3 a^2}/8$,

or $2-\sqrt{3}\leq a\leq 2/7$ and $(-4+16 a-4 a^2)/(3 a)\leq b\leq (16-3 a)/8 - \sqrt{3} \sqrt{64-32 a+3 a^2})/8$.

\end{enumerate}
Moreover, if the inequalities for the coefficients are strict, then the pairs are K-polystable.
\end{theorem}

The case of reducible boundary is interesting, because in this case we have several coefficients of boundary components, and we can look at the region of coefficients for which the pair is K-semistable (or K-polystable). 
Below we show these regions for the cases listed in Theorem \ref{main-thm}. In the case of irreducible boundary, we expect that K-polystability occurs more often. 

\vspace{1.5em}
\hspace{2em}
\includegraphics[width=3cm, height=3cm]{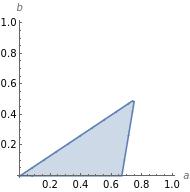}\quad
\includegraphics[width=3cm, height=3cm]{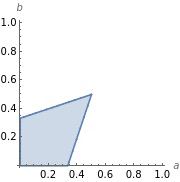}
\includegraphics[width=3cm, height=3cm]{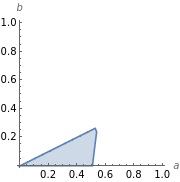}
\includegraphics[width=3cm, height=3cm]{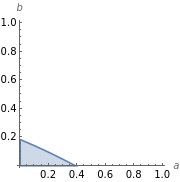}
\includegraphics[width=3cm, height=3cm]{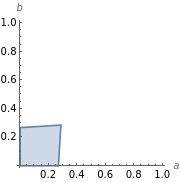}
\begin{center}
\textit{Figure $1$. Values of boundary coefficients for which log Fano pairs of Maeda type $(X, D)$ corresponding to the cases (i)--(v) as in Theorem \ref{main-thm} are K-semistable.}
\end{center}
\vspace{0.5em}
 
For all three-dimensional log smooth log Fano pairs $(X, \Delta=\sum D_i)$, we also determine the values of coefficients $c_i$ for which the corresponding pair $(X, D=\sum c_iD_i)$ is a log Fano pair of Maeda type, that is, $-K_X-D$ is ample. These regions of coefficients were studied in \cite{CMR20}, where they are called \emph{bodies of ample angles}. 
 
Initially, K-stability was defined in terms of one-parameter degenerations of a given variety and a Futaki invariant of such degenerations. However, in the works \cite{Li17}, \cite{Fu19}, a valuative criterion for K-stability was proposed. For convenience, we use it as a definition, see Definition \ref{def-k-stability}. So the problem is reduced to the computation of $\beta$-invariants associated to divisors over the given variety with respect to a pair. If such invariant is negative for some divisor, the pair is K-unstable. If one considers only divisors on the given variety, one comes to the notion of \emph{divisorial stability}, see Definition \ref{def-k-stability}. 

The main ingredient in the proof of Theorem \ref{main-thm} is the classification of threefold log smooth log Fano pairs $(X, \Delta=\sum D_i)$ due to Maeda \cite{M83}. We describe a natural minimal model program for such pairs, and a method to distinguish them. Also, we make some corrections to the Maeda classification. 
We summarise this in the following theorem. 

\begin{theorem}[{cf. \cite{M83}}]
\label{main-thm2}
Let $(X, \Delta)$ be a log smooth log Fano pair such that $\dim X = 3$. Then each step of an MMP on $K_X$ with scaling on $-K_X-\Delta$ is the blow down of $\mathbb{P}^2$ to a smooth point, and this MMP terminates on a Mori fiber space. 

Moreover, if $\Delta$ is reducible, then at most one step in this MMP is a birational contraction. All such pairs $(X, \Delta)$ belong to the Table \hyperlink{table-1}{1}. 
\end{theorem}

To prove Theorem \ref{main-thm}, we also need the Abban-Zhuang theory \cite{AZ20} which allows to work with K-stability inductively. In \cite{ACCFKMSSV21}, using this theory, some explicit estimates were obtained, see Proposition \ref{prop-formulas-for-beta}. We use them to check K-polystability.

The paper is organised as follows. In section \ref{sec-prelim} we recall the main definitions and facts on log smooth log Fano pairs. In section \ref{sect-classification} we recall the classification of threefold log smooth log Fano pairs, and describe an MMP with scaling for such pairs. In section \ref{sect-k-stability} we discuss K-stability and Abban-Zhuang theory. Then we start to consider log Fano pairs of Maeda type case by case according to the classification. In section \ref{sec-blow-up-of-quadric} we consider the blow up of a quadric. In section \ref{section-p2-bundles} we consider $\mathbb{P}^2$-bundles over $\mathbb{P}^1$. In section \ref{section-quadric-bundles} we consider quadric bundles over $\mathbb{P}^1$. In section \ref{section-p1-bundles} we consider $\mathbb{P}^1$ over log del Pezzo surfaces. In section \ref{sec-blow-up-of-p2-bundle} we consider the blow up of a point on certain $\mathbb{P}^1$-bundles over $\mathbb{P}^1$. In section \ref{sec-fano} we consider Fano threefolds. In section \ref{sec-proof-main} we prove theorems \ref{main-thm} and \ref{main-thm2}. Finally, in Appendix \ref{sect-table} we present a table that contains all the threefold log smooth log Fano pairs with reducible boundary. 

\textbf{Acknowledgements}. The author thanks Artem Avilov, Ivan Cheltsov, Aleksei Golota, and Constantin Shramov for useful discussions. The author is grateful to Kento Fujita who pointed out some inaccuracies in the first version of the paper and provided us with \cite{Fu11}, and to Yanir Rubinstein who suggested several improvements. The author thanks the referee for careful reading of the manuscript.

This work is supported by Russian Science Foundation under grant 21-71-00112.

\section{Preliminaries}
\label{sec-prelim}
We work over algebraically closed field of characteristic $0$. For the standard definitions of the minimal model program (MMP for short) we refer to \cite{KM98}. We recall the notion of a pair.

\begin{definition}
Let $X$ be a normal projective variety, and let $\Delta$ be an effective Weil $\mathbb{Q}$-divisor whose coefficients are between $0$ and $1$. We say that $(X, \Delta)$ is a \emph{pair} if $K_X + \Delta$ is $\mathbb{Q}$-Cartier. A pair $(X, \Delta)$ is called \emph{log smooth} if $X$ is smooth, and $\mathrm{Supp}(\Delta)$ has simple normal crossings. 
\end{definition}

\begin{definition}
\emph{Log Fano pair} is a log canonical pair $(X, \Delta)$ such that $-K_X-\Delta$ is ample. In this situation, $X$ is called a \emph{log Fano variety}. If, moreover, the pair $(X, \Delta)$ is log smooth, then we say that $X$ is a \emph{log Fano manifold}.
\end{definition}

Throughout the paper, we will assume that the divisor $\Delta$ is integral.


\begin{definition}[\cite{Fu20a}]
\emph{Log Fano pair of Maeda type} is a klt log Fano pair $(X, D)$ with $D\neq 0$ such that there exists a a log smooth log Fano pair $(X, \Delta)$ with integral $\Delta$, and $\mathrm{Supp}(D)\subseteq\mathrm{Supp}(\Delta)$.  
\end{definition}

We discuss some generalities on log smooth log Fano pairs. Let $(X, \Delta)$ be such pair. By the main result in \cite{Zha06} applied to the pair $(X, (1-\epsilon)\Delta)$ for some $0<\epsilon\ll 1$, the variety $X$ is rationally connected. Let $\Delta = \sum_{i=1}^{k} D_i$ where $D_i$ are prime divisors. By \cite[3.3]{Lo22}, $\Delta$ has no more than $\dim X$ components. In \cite{LM22} it is proven that if this bound is attained, then the pair $(X, \Delta)$ is toric. 
By section $2$ in \cite{M83}, we have $\Pic(X)\simeq \mathrm{H}^2(X, \mathbb{Z})$ and  $h^{i}(\oo_X) = 0$ for $i>0$. Also, $\Pic(X)$ is torsion free. Moreover, $\Pic(X)\otimes \mathbb{R} = \mathrm{N}^1(X)$ since the linear equivalence of divisors coincides with the numerical equivalence. The Mori cone $\overline{\mathrm{NE}}(X)$ is polyhedral and $(K_X+\Delta)$ is negative on $\overline{\mathrm{NE}}(X)\setminus\{ 0\}$. In particular, by Mori theory for any extremal face of $\overline{\mathrm{NE}}(X)$ there exists the corresponding contraction which contracts precisely the curves whose classes belong to that extremal face. For more details on the geometry of log smooth log Fano pairs, see \cite{Fu14b}.

There exists a description of the Picard group of a log smooth log Fano pair. For the definition of the Picard group of a reducible variety, we refer to \cite{Fu14b}.

\begin{theorem}[\cite{Fu14a}, 3.8]
\label{prop-pic-injective}
Assume that $(X, \Delta)$ is a log smooth log Fano pair with integral $\Delta\neq 0$. Then
\begin{enumerate}
\item
either $\mathrm{Pic}(X)\to\mathrm{Pic}(\Delta)$ is injective, or
\item
there exists a $\mathbb{P}^1$-bundle $\pi\colon X\to Y$ such that $\Delta$ is its section. In particular, $\Delta$ is irreducible and $\Delta\simeq Y$.
\end{enumerate}
\end{theorem}

\begin{corollary}
Let $(X,\Delta)$ be a $3$-dimensional log smooth log Fano pair with reducible boundary. Then $\rho(X)\leq3$.
\end{corollary}
\begin{proof}
According to \cite[2.8]{Fu14b}, the following sequence is exact:
\begin{equation}
\label{picard-exact}
0\to \mathrm{Pic}(\Delta) \xrightarrow{\nu} \bigoplus_i \mathrm{Pic}(D_i) \xrightarrow{\mu} \bigoplus_{i<j} \mathrm{Pic}(D_i\cap D_j)
\end{equation}
where $\nu$ is the restriction, and $\mu((\mathscr{L}_i)_i) = ((\mathscr{L}_i|_{D_i\cap D_j}\otimes\mathscr{L}_j|_{D_i\cap D_j}^{\vee})_{i,j})$.
By the classification of log smooth log del Pezzo surfaces \cite[3]{M83}, $\mathrm{Pic}(D_i)$ is isomorphic either to $\mathbb{Z}$, or to $\mathbb{Z}^2$. Note that $\Delta$ has no more than $3$ components, and $D_i\cap D_j\simeq \mathbb{P}^1$ for $i\neq j$, which follows either from the classification, or from \cite[3.3]{Lo22}. By Corollary \ref{prop-pic-injective}, we see that $\mathrm{Pic}(X)\to\mathrm{Pic}(\Delta)$ is injective. 
Analysing all possible cases using the classification of log smooth log del Pezzo pairs with integral boundary \cite[3]{M83}, we see that the rank of $\mathrm{Pic}(\Delta)$ is $\leq 3$ in each case, so the claim follows.
\end{proof}

\section{Classification}
\label{sect-classification}
We discuss the classification of three-dimensional log smooth log Fano pairs $(X, \Delta)$ with integral $\Delta\neq 0$ given in \cite{M83}. The approach of Maeda is based on the fact that there exists an extremal $K_X$-negative contraction $f\colon X\to X_1$ that satisfies the following property: for any contracted curve $C$ one has $(-K_X)\cdot C\geq -2$. Then, Maeda classifies pairs $(X, \Delta)$ according to the type of $f$. Note that one variety $X$ may admit several such contractions, hence it may appear in the classification several times: the lists given in \cite{M83} do not contain information on how to distinguish different log Fano pairs.

We propose a natural way to distinguish them. It is based on Maeda's ideas and a certain kind of MMP developed in the works of Casagrande \cite[2.1]{Ca12} and Fujita \cite{Fu14a}. We show that such contraction can be chosen in an (almost) canonical way: namely, as a first step of the MMP on $K_X$ with scaling on $-K_X-\Delta$. It is called a special MMP. Hence, the type of this contraction, as well as the value of the parameter $\epsilon$ (see Proposition \ref{prop-MMP}), become invariants of the pair $(X, \Delta)$. Thus we are able to distinguish such pairs in terms on these invariants.

We need a result that guarantees the existence of a $1$-complement with some special properties. 

\begin{proposition}[cf. \cite{LM22}, Lemma 5.6]
\label{prop-1-complement}
For a log smooth log Fano pair $(X, \Delta)$ with integral~$\Delta$, there exists a $1$-complement, that is, an integral effective divisor $\Gamma$ such that $K_X+\Delta+\Gamma\sim 0$ and the pair $(X, \Delta + \Gamma)$ is lc. Moreover, we have $\Gamma|_\Delta=\Gamma_{\Delta}$ where $\Gamma_{\Delta}$ is a $1$-complement on $\Delta$.
\end{proposition}
\begin{proof}
Put $\Delta=\sum_i D_i$. We prove the proposition by induction. The case $\dim X = 1$ is clear. Assume that $\dim X\geq 2$. Note that the pair $(D_i, \Delta_{D_i})$ is a log smooth log Fano pair where $\Delta_{D_i}=\sum_{j\neq i} D_j|_{D_i}$, see e.g. \cite[3.3]{Lo22}. Moreover, each stratum $\Theta=D_{i_1}\cap \ldots \cap D_{i_{m}}$ of $\Delta$ is a log smooth log Fano manifold with respect to the boundary $\Delta_\Theta = \sum_{j \neq i_1, \ldots, i_{m}} D_j|_{\Theta}$.  
By induction, there exists a $1$-complement $\Gamma_{D_i}$ on each component $D_i$ such that the pair $(D_i, \Delta_{D_i}+\Gamma_{D_i})$ is lc, and $\Gamma_{D_i}|_{\Delta_{D_i}} = \Gamma_{\Delta_{D_i}}$. Moreover, we may assume that $\Gamma_{D_i}|_{D_i\cap D_j}=\Gamma_{D_j}|_{D_i\cap D_j}=\Gamma_{D_i\cap D_j}$ for any $i\neq j$. Using the exact sequence \eqref{picard-exact}, wee see that a $1$-complement $\Gamma_{\Delta}$ of $K_{\Delta}$ is defined on $\Delta$. By Kawamata-Viehweg vanishing, it can be lifted to an integral divisor $\Gamma$ on $X$ such that of $K_X + \Delta + \Gamma\sim0$. 

By inversion of adjunction the pair $(X, \Delta + \Gamma)$ is lc near $\Delta$. Then, if it is not lc, for some $\lambda>0$ the pair $(X, \Delta + (1-\lambda)\Gamma)$ is lc, and its non-klt locus has $\geq 2$ connected components. This contradicts Koll\'ar-Shokurov connectedness theorem since
$
- K_X - \Delta - (1-\lambda)\Gamma\sim_{\mathbb{Q}} \lambda (-K_X-\Delta)
$
is ample.
\end{proof}

\subsection{Special MMP} 
Now, let $(X, \Delta)$ be a log smooth log Fano pair (of any dimension) with integral $\Delta\neq 0$. Let $\Gamma\sim-K_X-\Delta$ be a $1$-complement as in Proposition \ref{prop-1-complement}. Note that a log Fano pair $(X, \Delta)$ is a Mori dream space by \cite[1.3.2]{BCHM10}. We describe a $K_X$-MMP with scaling on~$\Gamma$. 

\begin{proposition}[{cf. \cite[2.1]{Ca12}, \cite[Section $3$]{Fu14a}}]
\label{prop-MMP}
The $K_X$-MMP with scaling on $\Gamma$ is also a $-\Delta$-MMP and $-\Gamma$-MMP. In particular, it terminates on a Mori fiber space. In dimension $3$, each step of this MMP is a contraction of $\mathbb{P}^2$ to a smooth point.
\end{proposition}
\begin{proof}
Put $X_0=X$, $\Delta_0=\Delta$. We run $K_{X_0}$-MMP with scaling on $\Gamma_0=-K_{X_0}-\Delta_0$. Note that $K_{X_0}$ is not nef.  Indeed, otherwise $K_{X_0} + \Gamma_0 \sim -\Delta_0\neq 0$ is nef which is absurd. Pick minimal $\epsilon_0>0$ such that $K_{X_0} + \epsilon \Gamma_0$ is nef. According to Mori theory, $\epsilon_0$ is rational and there exists an extremal ray $R_0$ with the property 
$
(K_{X_0} + \epsilon_0 \Gamma_0)\cdot R_0 = 0.
$ 
Let $f_0\colon X_0\to Z_1$ be the contraction associated to $R_0$. If it is of fiber type, we stop. Assume that it is birational. If it is divisorial, put $X_1=Z_1$, and $g_0=f_0$. If it is small, put $X_1 = X_0^+$ where $X_0\to Z_1\leftarrow X_0^+$ is a flip, and put $g_0\colon X_0\dashrightarrow X_1$. 

Since $K_{X_0}$ is not nef and $K_{X_0} + \Gamma_0 \sim -\Delta_0$ is not nef as well, we have $\epsilon_0>1$ since $K_{X_0} + \epsilon \Gamma_0$ is nef for $\epsilon\gg0$. We have
$
((1-\epsilon_0)K_{X_0} - \epsilon_0 \Delta_0)\cdot R_0 = 0,
$
and since $\epsilon_0-1>0$, $K_{X_0}\cdot R_0<0$, it follows that $\Delta_0\cdot R_0>0$. Hence, $g_0$ is a step of $-\Delta$-MMP. Moreover, since $\Gamma_0\cdot R_0>0$, we see that $g_0$ is a step of $-\Gamma$-MMP. Also, we have 
\begin{equation}
\label{prop-mmp-length}
(-K_{X_0})\cdot R_0 = \Gamma_0\cdot R_0 + \Delta_0\cdot R_0\geq 2.
\end{equation}

Put $\Delta_{1}=(g_0)_*\Delta_0$, $\Gamma_{1}=(g_0)_*\Gamma_0$. Note that these divisors are non-zero since $g_0$ is a step of $-\Delta$-MMP and $-\Gamma$-MMP.
One checks that $K_{X_{1}} + \epsilon_0 \Gamma_{1}$ is nef, but  $K_{X_{1}}$ is not nef (otherwise $K_{X_{1}}+\Delta_1\sim -\Gamma_1\neq 0$ would be pseudo-effective which is absurd). Set $\epsilon_{1}$ to be the minimal number such that $K_{X_{1}} + \epsilon_{1} \Gamma_{1}$ is nef. We have $0<\epsilon_{1}\leq \epsilon_0$, and $\epsilon_1\in \mathbb{Q}$. Note that $K_{X_{1}} + \epsilon_{1} \Gamma_{1}$ is nef but not ample. Hence there exists an extremal ray $R_1$ such that 
$
(K_{X_{1}} + \epsilon_{1} \Gamma_{1})\cdot R_{1}=0.
$
Similarly to the above, we define $g_1\colon X_1\to X_2$ which is also a step of $-\Delta$-MMP and $-\Gamma$-MMP. Iterating, we get a sequence of maps
\[
X = X_0 \dashrightarrow X_1 \dashrightarrow \ldots \dashrightarrow X_k \to Z
\]
where $g_i\colon X_{i}\dashrightarrow X_{i+1}$ for $0\leq i\leq k-1$ is birational, and $\pi \colon X_k\to Z$ is of fiber type. 

In dimension $3$, by \eqref{prop-mmp-length} we have $(-K_{X_0})\cdot R_0\geq 2$, and since $X_0$ is smooth, by the classification of $K_X$-negative contractions on smooth threefolds, we get that $g_0$ is a contraction of $\mathbb{P}^2$ to a smooth point. Arguing analogously and using that $\Gamma_i$ is an integral non-zero divisor for each $i$, we see that each $g_i$ for $0\leq i\leq k-1$ is also a contraction of $\mathbb{P}^2$ to a smooth point.
\end{proof}

\begin{remark}
\label{rem-nef-value}
In the literature (see e.g. \cite{BS97}), the value $\epsilon=\epsilon_0$ is sometimes called the \emph{nef value}, and the  morphism given by a multiple of the linear system $|K_X+\epsilon L|$ is called the \emph{nef value morphism}. Note that a nef value morphism needs not to be extremal. However, if it is extremal, then its type becomes an invariant of the pair $(X, \Delta)$.
\end{remark}

\begin{corollary}[{cf. \cite[Section 5.1]{M83}}]
\label{X-classification}
Assume that $(X, \Delta)$ is a $3$-dimensional log smooth log Fano pair such that $\Delta$ is integral. Assume further that $\Delta$ is reducible. Then there are the following possibilities for $X$:
\begin{enumerate}
\item
the blow up of a quadric $Q$ in a point, see case \hyperref[sec-blow-up-of-quadric]{E1} in Table \hyperlink{table-1}{$1$},
\item
it admits the structure of an extremal $\mathbb{P}^2$-bundle, see cases \hyperlink{compute-p2-bundle-cases-remark}{D1}--\hyperlink{compute-p2-bundle-cases-viii}{D8} in Table \hyperlink{table-1}{$1$},
\item
it admits the structure of an extremal quadric bundle, see case \hyperlink{compute-quadric-bundle-case-1}{Q1} in Table \hyperlink{table-1}{$1$},
\item 
it admits the structure of an extremal $\mathbb{P}^1$-bundle, see cases \hyperlink{compute-p1-bundle-cases-remark}{C1}--\hyperlink{compute-p1-bundle-cases-9}{C10} in Table \hyperlink{table-1}{$1$},
\item
the blow up of a point $p$ on a $\mathbb{P}^2$-bundle of the form $\mathbb{P}_{\mathbb{P}^1}(\oo\oplus\oo\oplus\oo(n))$ for $n\geq 1$, see case \hyperref[sec-blow-up-of-p2-bundle]{E2} in Table \hyperlink{table-1}{$1$},
\item
\label{X-is-Fano}
a Fano threefold with $\rho(X)=1$ and $i(X)\geq 3$, that is, a quadric $Q$ or a projective space $\mathbb{P}^3$, see cases \hyperref[subsec-p3]{F1}--\hyperref[subsec-quadric]{F4} in Table \hyperlink{table-1}{$1$}.
\end{enumerate}
\end{corollary}
\begin{proof}
The proof relies on the classification. First assume that the $K_X$-MMP as in Proposition \ref{prop-MMP} does not terminate on the first step. Then $g_1\colon X\to X_1$ is the blow down of a projective plane to a smooth point on $X_1$. Denote by $E\simeq \mathbb{P}^2$ the $g_1$-exceptional divisor. According to Maeda's classification, see section 5.1 in \cite{M83}, there are the following possibilities: 
\begin{enumerate}[leftmargin=*, label*=\arabic*.]
\item
Case 5.1(i) as in \cite{M83}, $E$ is a component of $\Delta=D_1+D_2$, where $D_1 = E \simeq \mathbb{P}^2$ and $D_2\simeq \mathbb{F}_2$. Here $g_1\colon X\to X_1$ is the blow up of a quadric $X_1$ in a point. 
This case is treated in section \ref{sec-blow-up-of-quadric}.
\item
Case 5.1(ii), $E$ is a component of $\Delta=D_1+D_2+D_3$ where $D_1 = E \simeq \mathbb{P}^2$, $D_2\simeq D_3\simeq \mathbb{F}_1$. Here $X$ is the blow up of a point on $\mathbb{P}^3$. One checks that in this case the nef value is equal to $2$, and the nef value morphism is a contraction to a point, see Example 4.6 in \cite{LM22}. Hence, the first step as the special MMP is not uniquely determined: it can be either the blow down of $E$, or a $\mathbb{P}^1$-bundle over $\mathbb{P}^2$. We treat it as a special case of a more general case of $\mathbb{P}^1$-bundles, see \hyperlink{compute-p1-bundle-cases-ii}{C2}. 
\item
Case 5.1(iv), $E$ is not a component of $\Delta=D_1+D_2$ where $D_1\simeq \mathbb{F}_1$, and $D_2\simeq \mathbb{P}^2$. Moreover, $E\cap D_1$ is a $(-1)$-curve. Here $X$ is the blow up of $\mathbb{P}^3$ in a point, $D'_1$ is a plane that passes through the blown up point, and $D'_2$ is a plane that does not intersect it. Again, $X$ admits an extremal $\mathbb{P}^1$-bundle structure, so we to treat this case as a special case of \hyperlink{compute-p1-bundle-cases-iii}{C3}. 
\item
\label{X-is-Fano-blow-up-of-P2-bundle}
Case 5.1(iv), $E$ is not a component of $\Delta=D_1+D_2$, where $D_1\simeq\mathbb{F}_1$, and $D_2\simeq \mathbb{P}^1\times\mathbb{P}^1$. Here $X$ is the blow up of a point on a $\mathbb{P}^2$-bundle $Y$ of the following form: 
\[
Y=\mathbb{P}_{\mathbb{P}^1}(\oo\oplus\oo\oplus\oo(n)), \ \ \ \ \ n>0, \ \ \ \ \ \Delta_Y=D'_1+D'_2, \ \ \ \ \ D'_1\sim F, \ \ \ \ \ D'_2\sim H-n F 
\]
where $F$ is a fiber of the projection $Y\to \mathbb{P}^1$, and $H$ is the tautological divisor. We treat this case in section \ref{sec-blow-up-of-p2-bundle}.
\end{enumerate}

Now we assume that the MMP terminates on the first step. Hence, there exists a Mori fiber space structure $\pi\colon X\to Z$. Then, there are the following possibilities:
\begin{enumerate}[leftmargin=*, label*=\arabic*.]
\item
Case 5.4 as in \cite{M83}, $\dim Z = 0$. Then $X$ is a smooth Fano threefold with $\rho(X)=1$. Note that for the Fano index of $X$ we have $i(X)\geq 2$, because $-K_X-\Delta$ is ample. Since by assumption $\Delta$ is reducible, we have $i(X)\geq 3$, so $X$ is either a quadric, or a projective space $\mathbb{P}^3$. This case is treated in section \ref{sec-fano}.
\item
Case 5.3, $\dim Z = 1$. Then $Z=\mathbb{P}^1$, and $\pi\colon X\to Z$ is a del Pezzo fibration. Note that the general fiber has index $\geq 2$, hence it is isomorphic either to a projective plane $\mathbb{P}^2$, or to a quadric $\mathbb{P}^1\times \mathbb{P}^1$. We consider these cases in sections \ref{section-p2-bundles} and \ref{section-quadric-bundles}, respectively. 
\item
Case 5.2, $\dim Z = 2$. Then $X\to Z$ is a $\mathbb{P}^1$-bundle. We consider this case in section \ref{section-p1-bundles}.
\end{enumerate}
\end{proof}

\section{K-stability}
\label{sect-k-stability}
In this section, we recall the main definitions and theorems from the theory of K-stability. For a detailed survey, see \cite{Xu21}. 
\begin{definition}
Let $(X, D)$ be a pair, and let $f\colon Y\to X$ be a proper birational morphism from a normal variety $Y$. For a prime divisor $E$ on $Y$, a \emph{log discrepancy} of $E$ with respect to the pair $(X, D)$ is defined as
\[
A_{(X, D)}(E) = 1 + \mathrm{coeff}_E \Big( K_Y - f^*(K_X + D) \Big).
\]
\end{definition}

Put $L=-K_X-D$. By a \emph{pseudo-effective threshold} of $E$ with respect to a log Fano pair $(X, D)$ we mean the number
\[
\tau_{(X, D)}(E) = \sup\{ x\in\mathbb{R}_{\geq0}\colon f^*L - xE\ \text{is pseudo-effective}\}.
\]

Similarly, we define the \emph{nef threshold} of $E$ with respect to a log Fano pair $(X, D)$:
\[
\epsilon_{(X, D)}(E) = \sup\{ x\in\mathbb{R}_{\geq0}\colon f^*L - xE\ \text{is nef}\}.
\]

The \emph{expected vanishing order} of $E$ with respect to a log Fano pair $(X, D)$ is 
\[
S_{(X, D)}(E) = \frac{1}{\mathrm{vol}(L)} \int_{0}^{\infty}{\mathrm{vol}\Big( f^*L - xE \Big)dx},
\]
where $\mathrm{vol}$ is the volume function, see \cite{Laz04}. We will also use the following notation: $S'_{(X, D)}(E)={\mathrm{vol}(L)}\cdot S_{(X, D)}(E)$. The \emph{beta-invariant} $\beta_{(X, D)}(E)$ of $E$ with respect to a log Fano pair $(X, D)$ is defined as follows:
\[
\beta_{(X, D)}(E) = A_{(X, D)}(E) - S_{(X, D)}(E).
\]

We also use the following notation:
\[
\beta'_{(X, D)}(E) = \mathrm{vol}(L)\cdot \beta_{(X, D)}(E) = \mathrm{vol}(L)\cdot A_{(X, D)}(E) - \int_{0}^{\infty}{\mathrm{vol}\Big( f^*L-tE \Big)dx}.
\]

Recall that the $\delta$-invariant of $E$ with respect to a log Fano pair $(X, D)$ (resp., $\delta$-invariant of $E$ along the subvariety $Z$ with respect to a log Fano pair $(X, D)$) are defined as
\[
\delta(X, D) = \inf_{E/X} \frac{A_{(X, D)}(E)}{S_{(X, D)}(E)}, \ \ \ \ \delta_Z(X, D) = \inf_{E/X,\ Z\subset C_X(E)} \frac{A_{(X, D)}(E)}{S_{(X, D)}(E)}
\]
where $E$ runs through all prime divisors over $X$ (resp., $E$ runs through all prime divisors over $X$ whose center contains $Z$). We formulate the main definitions.

\begin{definition}[\cite{Li17}, \cite{Fu19}, \cite{Fu16a}]
\label{def-k-stability}
A klt log Fano pair $(X, D)$ is called
\begin{enumerate}
\item
\emph{divisorially semistable} (resp., \emph{divisorially stable}), if $\beta_{(X, D)}(E)\geq0$ (resp., $\beta_{(X, D)}(E)>0$) for any prime divisor $E$ on $X$. We say that $X$ is \emph{divisorially unstable} if it is not divisorially semistable,
\item
\emph{K-semistable} (resp., \emph{K-stable}) if $\beta_{(X, D)}(E)\geq0$ (resp., $\beta_{(X, D)}(E)>0$) for any prime divisor $E$ over $X$. We say that $X$ is \emph{K-unstable} if it is not K-semistable,
\item
\emph{uniformly K-stable} if $\delta(X, D)>1$.
\end{enumerate}
\end{definition}

This definition is equivalent the original definition of K-stability given in \cite{Tia97}, \cite{Don02} in terms of test configurations. We refer to \cite[2.7]{Fu19} for the definition of K-polystability. 

In what follows, for convenience we will compute $\beta'_{(X, D)}$ instead of $\beta_{(X, D)}$ to characterise K-stability. To analyse K-semistability inductively, we need the Abban-Zhuang theory developed in \cite{AZ20}. This theory was applied to get explicit results in the case of Fano threefolds in \cite[1.7]{ACCFKMSSV21}. We need these results in a slightly more general form as we work with log Fano pairs.

 
\begin{proposition}[cf. {\cite[1.7.26]{ACCFKMSSV21}}]
\label{prop-formulas-for-beta}
Let $(X, D)$ be a $3$-dimensional log Fano pair of Maeda type that satisfies $\overline{\mathrm{Mov}}(X)=\mathrm{Nef}(X)$. Let $Y$ be an irreducible normal surface in $X$, let $Z$ be an irreducible curve in $Y$, and let $E$ be a prime divisor over $X$ such that $C_X(E) = Z$. Put $L=-K_X-D$. Then
\begin{equation}
\label{eq-min}
\frac{A_{(X, D)}(E)}{S_{(X, D)}(E)}\geq \mathrm{min} \Bigg\{ \frac{A_{(X, D)}(Y)}{S_{(X, D)}(Y)}, \frac{A_{(Y, D_Y)}(Z)}{S(W^Y_{\bullet, \bullet}; Z)} \Bigg\}
\end{equation}
with
\begin{equation}
\label{eq-min2}
S(W^Y_{\bullet, \bullet}; Z) = \frac{3}{L^3}\int_0^{\infty}{\big(P(u)^2\cdot Y\big)\cdot \mathrm{ord}\Big(N(u)|_Y\Big) du} + 
\frac{3}{L^3} \int_0^{\infty}\int_0^{\infty}{\mathrm{vol}\Big(P(u)|_Y - v Z\Big)dv du},
\end{equation}
where $P(u)$ is the positive part of the Zariski decomposition of $L-uY$, and $N(u)$ is its negative part. 
\end{proposition}
\begin{proof}
We use the notation as in \cite{AZ20} and \cite[1.7]{ACCFKMSSV21}. Put $V_{\overrightarrow{\bullet}} = V_\bullet = \bigoplus_{m\geq 0}V_m = \bigoplus_{m\geq 0} \mathrm{H}^0(X, \oo_X(-mL))$, hence by definition  
$
\delta_Z(X, D, V_\bullet)=\delta_Z(X, D).
$
Put 
\[
W_{\overrightarrow{\bullet}} = W_{\bullet, \bullet} = \bigoplus_{m\geq 0, j\geq 0} W_{m, j} = \bigoplus_{m\geq 0, j\geq 0} \mathrm{im} \bigg( \mathrm{H}^0 (X, mL-jY) \to \mathrm{H}^0 (Y, mL_Y-jY|_Y) \bigg)
\]
where the map on the right is the natural restriction homomorphism, and $L_Y=L|_Y$. We apply Theorem 3.3 in \cite{AZ20} to an irreducible normal surface $Y\subset X$. Define the pair $(Y, D_Y)$ by adjunction: $K_X + D + (1-\mathrm{coeff}_D Y)Y|_Y = K_Y + D_Y$. Then, according to this theorem, we have
\begin{equation}
\begin{split}
\delta_Z(X, D) \geq \min \bigg\{  \frac{A_{(X, D)}(Y)}{S(V_\bullet, Y)}, \ \delta_Z(Y, D_Y, W_{\overrightarrow{\bullet}}) \bigg\}
= \min \bigg\{  \frac{A_{(X, D)}(Y)}{S_{(X, D)}(Y)}, \ \delta_Z(Y, D_Y, W_{\overrightarrow{\bullet}}) \bigg\}.
\end{split}
\end{equation}

We have $\delta_Z(Y, D_Y, W_\bullet) = \frac{A_{(Y, D_Y)}(Z)}{S_{(Y, D_Y)}(Z)}$, and the formula \eqref{eq-min} follows. To compute $S_{(Y, D_Y)}(Z)$, we apply Corollary 1.7.24 in \cite{ACCFKMSSV21} which uses the assumption $\overline{\mathrm{Mov}}(X)=\mathrm{Nef}(X)$. According to this corollary, we have
\[
S_{(Y, D_Y)}(Z) = \frac{3}{\mathrm{vol}(L)} \int_0^{\infty}{h(u)du}\quad \text{for}\quad h(u) = \big(P(u)^2\cdot Y\big)\cdot \mathrm{ord}\Big(N(u)|_Y\Big) + \int_0^\infty{\mathrm{vol}\Big(P(u)|_Y - v Z\Big)dv},
\]
so the formula \eqref{eq-min2} is proven. This concludes the proof.
\end{proof}

The following theorem characterises two-dimensional K-polystable log Fano pairs of Maeda type.

\begin{theorem}[{\cite[Corollary 1.2]{Fu20a}}]
\label{thm-2-dim}
Let $(X, D)$ be a log del Pezzo pair of Maeda type. Then $(X, D)$ is K-semistable (resp., K-polystable) if and only if $(X, D)$ is isomorphic to 
\begin{itemize}
\item
 $(\mathbb{P}^2, aC)$ with $C$ a smooth conic and $a \leq 3/4$ (resp., $a<3/4$), or
\item
$(\mathbb{P}^1\times \mathbb{P}^1, aC)$ with $C$ the diagonal and $a \leq 1/2$ (resp., $a<1/2$).
\end{itemize}
\end{theorem}

In any dimension, it is known that if the coefficients of the boundary are sufficiently close to~$1$, then a log smooth log Fano pair is K-unstable.

\begin{theorem}[{\cite[Corollary 1.2]{Fu16b}}]
\label{thm-G-small-values}
Let $X$ be a smooth projective variety and $D$ be a nonzero reduced simple normal crossing divisor on $X$. Assume that $-K_X -(1-\beta)D$ is ample for any $0 < \beta \ll 1$ and the divisor $-K_X - D$ is big. Then $(X, -K_X - (1 - \beta)D)$ is not K-semistable for any $0 <\beta \ll 1$ with $\beta \in \mathbb{Q}$. 
\end{theorem}

The following theorem allows to reduce the computations to the $G$-equivariant case.

\begin{theorem}[{\cite[Corollary 4.14]{Zhu20a}}]
\label{thm-G-invariant}
Let $G$ be an algebraic group and let $(X, D)$ be a log Fano pair with a $G$-action. Assume that $A_{(X, D)}(E) \geq S_{(X, D)}(E)$ (resp., $G$ is reductive and $A_{(X, D)}(E) > S_{(X, D)}(E)$) for all $G$-invariant irreducible divisors $E$ over $X$. Then $(X, D)$ is K-semistable (resp. K-polystable).
\end{theorem}

\section{Blow up of a quadric in a point} 
\label{sec-blow-up-of-quadric}
In this section, we consider a log Fano pair $(X, D)$ of Maeda type where $X$ is the blow up of a point on a smooth quadric $Q\subset \mathbb{P}^4$. Then $X$ is a Fano variety {2.30} as in \cite[12.3]{IP99}. Note that by \cite{Fu16a} the Fano variety $X$ is divisorially unstable. First we recall some elementary facts concerning its geometry. Consider the following diagram
\begin{equation*}
\vcenter{
\xymatrix{
& \ar[dl]_{f_1} \mathrm{Bl}_pQ = X \ar[dr]^{f_2}  &    
\\
\mathbb{P}^4\supset Q = X_1 &  & \mathbb{P}^3 = X_2
}}
\end{equation*}
where $f_1$ is the blow up of a point $p\in Q$, and $f_2$ is the blow up of a plane conic $q\subset \mathbb{P}^3$. Let $D_2'=T_pQ\cap Q\simeq \mathbb{P}(1,1,2)$ be the intersection of the tangent hyperplane with the quadric, and let $D_1'\subset X_2$ be a plane that contains the conic $q$. Put $D_1=(f_2^{-1})_*(D_1')$ and $D_2=(f_1^{-1})_*(D_2')$.

\begin{proposition}
\label{prop-unstable-blow-up-of-quadric}
The pair $(X, D)$ as above, where $D = aD_1 + bD_2$ for any $0\leq a,b<1$, is divisorially unstable.
\end{proposition}
\begin{proof}

We have $D_1 \simeq \mathbb{P}^2$, $D_2\simeq \mathbb{F}_2$, and $D_1\cap D_2$ is a smooth conic on $D_1$ and the $(-2)$-curve on $D_2$. Let $l_1$ be a line in $D_1$, and $l_2$ be a ruling in $D_2$. We have $l_1\cdot D_1=-1,\ l_1\cdot D_2 = 2,\ l_2\cdot D_1 = 1,\ l_2\cdot D_2 = -1$, and
$
\Gamma=-K_X - \Delta \sim 3D_1 + 2D_2,
$
Note that $\Gamma\cdot l_1 = \Gamma\cdot l_2 = 1$. This shows that $(X, \Delta=D_1+D_2)$ is a log Fano pair. 
Compute the intersection theory on $X$: 
$
D_1^3 = 1,\
D_1^2\cdot D_2 = - 2,\
D_1 \cdot D_2^2 = 4,\
D_2^3 = -6.
$
One checks that $\overline{\mathrm{NE}}(X)$ is generated by $l_1$ and $l_2$, and
\begin{equation*}
\begin{split}
\mathrm{Nef}(X) &= \mathbb{R}_{\geq 0} \big[2 D_1 + D_2\big] + \mathbb{R}_{\geq 0} \big[D_1 + D_2\big], \ \ \ \
\overline{\mathrm{Eff}}(X) = \mathbb{R}_{\geq 0} \big[D_1\big] + \mathbb{R}_{\geq 0} \big[D_2\big]. 
\end{split}
\end{equation*}
Note that $\overline{\mathrm{Mov}}(X)=\mathrm{Nef}(X)$. 
We have $K_X\cdot l_1=-2$, $K_X\cdot l_2=-1$, so for the nef value we have $\epsilon=2$, and the nef value morphism (see Remark \ref{rem-nef-value}) is the contraction $f_1$. 
Note that $L=-K_X-D \sim ( 4 - a ) D_1 + ( 3 - b ) D_2$ is ample for $0\leq a,b<1$. 
We estimate $\beta'_{(X, D)}(D_2)$. By definition $A_{(X, D)}(D_2) = 1 - b$. We have
$
L - xD_2 \sim \big( 4 - a \big) D_1 + \big( 3 - b - x \big) D_2 
= \big( 2 - a/2 \big) \big( 2D_1+D_2 \big) + \big( 1 + a/2 - b - x \big) D_2.
$
Then $\tau_{(X, D)}(D_2)=3-b$, and $\epsilon_{(X, D)}(D_2)=1+a/2-b$.
Compute 
\begin{multline*}
\int_0^{1+a/2-b}{\mathrm{vol}\big(L-xD_2\big)dx} = 
\int_0^{1+a/2-b}{\bigg(\big( 4 - a \big) D_1 + \big( 3 - b - x \big) D_2\bigg)^3dx} \\
=\big(4-a\big)^3\big(a/2-b+1\big) + 3\big(4-a\big)^2\big(a/2-2\big)^2 - 3\big(4-a\big)^2\big(b-3\big)^2 \\
+ 4\big(4-a\big)\big(a/2-2\big)^3 - 4\big(4-a\big)\big(b-3\big)^3 + 3\big(a/2-2\big)^4/2 - 3\big(b-3\big)^4/2.
\end{multline*}

We have  
\begin{multline*}
\beta'_{(X, D)}(D_2) \leq A_{(X, D)}(D_2)L^3 - \int_0^{1+a/2-b}{\mathrm{vol}\big(L-xD_2\big)dx} \\
= \big(1-b\big) \big(4-a\big)^3 - 6 \big(1-b\big) \big(4-a \big)^2 \big(3-b\big) + 12 \big(1-b\big) \big(4-a\big) \big(3-b\big)^2 - 6 \big(1-b\big) \big(3-b\big)^3 \\
-  (4-a )^3 (a/2-b+1 ) - 3 (4-a )^2 (2-a/2 )^2 - 3 (4-a )^2 (3-b )^2 \\
+ 4 (4-a ) (2-a/2 )^3 - 4 (4-a ) (3-b )^3 - 3 (2-a/2 )^4/2 + 3 (3-b )^4/2 \\
= \big(4-a\big)\big(3-b\big)\big(1-b\big)\big( -12b+9a\big) - 6\big(4-a\big)^2\big(3-b\big) \\
+ \big(3-b\big)^3\big( -35/2 + 9b/2 + 4a \big) - \big(22+5a/2\big)\big(2-a/2\big)^3/2 \\
\leq 4\cdot 3\cdot 9 - 6\cdot9\cdot 2-8\cdot 9 -11\cdot 27/2<0.
\end{multline*}
Hence, the pair is K-unstable.
\end{proof}

\section{$\mathbb{P}^2$-bundles over $\mathbb{P}^1$}
\label{section-p2-bundles}
In this section, we consider log Fano pairs $(X, D)$ of Maeda type such that $X$ has the structure of an extremal $\mathbb{P}^2$-bundle over $\mathbb{P}^1$. Then
\[
X = \mathbb{P}_{\mathbb{P}^1}\big(\oo\oplus\oo(k)\oplus\oo(n)\big), \ \ \ \ \ \ \ 0\leq k \leq n.
\]
Denote by $F$ a fiber of the projection $\pi\colon X\to \mathbb{P}^1$, and by $H$ the tautological divisor on $X$. Note that the linear system $|H|$ is base point free and defines a contraction on $X$. 
It is easy to check that
\begin{equation*}
\begin{split}
K_X &\sim -3H - (2-k-n)F, \ \ \ \ \ \ H^3 = k + n, \ \ \ \ \ \ H^2F=1, \\
\mathrm{Nef}(X) &= \mathbb{R}_{\geq 0} \big[H\big] + \mathbb{R}_{\geq 0} \big[F\big], \ \ \ \ \ \ 
\overline{\mathrm{Eff}}(X) = \mathbb{R}_{\geq 0} \big[H-nF\big] + \mathbb{R}_{\geq 0} \big[F\big]. \\
\end{split}
\end{equation*}

According to \cite[5.3, 9.1]{M83}, there are the following possibilities for $(X, \Delta)$ in the case of reducible~$\Delta=\sum D_i$. They are denoted \hyperlink{compute-p2-bundle-cases-remark}{D1} --\hyperlink{compute-p2-bundle-cases-viii}{D8} in Table 1.

\begin{remark}
\label{rem-extra-case1}
Note that there is also the following pair $(X, \Delta=D_1+D_2)$ considered in \cite[9.1.3\,(vi)]{M83}. Here $X=\mathbb{P}_{\mathbb{P}^1}(\oo\oplus\oo(1)\oplus\oo(1))$, $D_1\sim H$, $D_2\sim H$. However, it is not a log Fano pair. Indeed, for a smooth rational curve $l = (H-F)^2$ which is an exceptional curve of a small contraction on $X$, we have $H\cdot l = 0$. This shows that the divisor $-K_X-\Delta\sim H$ is not ample.
\end{remark}


One checks that in all the cases the nef value morphism is extremal and coincides with $\pi$. In this section, we prove the following
\begin{proposition}
\label{prop-unstable-p2bundles}
Let $(X, D)$ be a log Fano pair of Maeda type where $(X, \Delta)$ is as in Table $1$. Then $(X, D)$ is divisorially unstable unless $(X, D)$ is as in the case \hyperlink{compute-p2-bundle-cases-v}{D5}.
\end{proposition}
\hypertarget{compute-p2-bundle-cases-remark}
Note that the case \hyperlink{compute-p2-bundle-cases-remark}{D1}  can be treated as a special case of \hyperlink{compute-p2-bundle-cases-viii}{D8} with the value of the parameter $c$ equal to $0$. The case \hyperlink{compute-p2-bundle-cases-remark}{D3} can be treated as a special case of \hyperlink{compute-p2-bundle-cases-ii}{D2} for $n=0$. The case \hyperlink{compute-p2-bundle-cases-remark}{D6} can be treated as a particular case of \hyperlink{compute-p2-bundle-cases-viii}{D8} with $k=0$ and the value of the parameter $a$ equal to $0$. 




\

\textbf{Case} D2. 
\hypertarget{compute-p2-bundle-cases-ii}
We show that in this case the log Fano pair $(X, D)$ is divisorially unstable for any values of the parameters. We have
\[
X = \mathbb{P}_{\mathbb{P}^1}\big(\oo\oplus\oo(1)\oplus\oo(n)\big), \ \ \ \ \ n\geq 0, \ \ \ \ \ D_1\sim H, \ \ \ \ \ D_2\sim H - nF.
\]
Put $D = aD_1 + bD_2$, $0\leq a, b<1$. Then  
$
L\sim-K_X-D \sim \big(3 - a - b\big) H + \big( 1 - n +  b n \big)F 
=\big(3 - a - b\big) \big( H - n F \big) + \big( 1 + 2n - a n \big)F
$
is ample if and only if $n(1-b) < 1$. Compute $\beta'_{(X, \Delta')}(D_1)$. We have 
\begin{equation*}
\begin{split}
S'_{(X, D)}(D_1)
=&\int_0^{3-a-b}{\bigg(\big(3 - a - b - x\big) H + \big( 1 - n + b n \big)F \bigg)^3dx} \\
=& \big(n + 1\big) \big(3-a-b\big)^4/4 + \big(3-a-b\big)^3 \big( 1 - n + b n \big).
\end{split}
\end{equation*}
Then $
\beta'_{(X, D)}(D_1) 
= \big(3 - a - b\big)^3\big(n+1\big) \big( 1 - 3a + b \big)/4 + \big(3 - a - b\big)^2\big( 1 - n + b n \big) \big( b - 2a  \big). 
$ Compute $\beta'_{(X, D)}(D_2)$. We have $A_{(X, D)}(D_2) = 1 - b$. It is easy to see that $L-xD_2$ for $x\geq0$ is nef if and only if it is pseudo-effective if and only if $x\leq 3-a-b$.
Compute
\begin{equation*}
\begin{split}
S'_{(X, D)}(D_2)
=& \int_0^{3-a-b}{\bigg(\big(3 - a - b - x\big) \big( H - n F \big) + \big( 1 + 2n - a n \big)F \bigg)^3dx} \\
=& \big(1-2n\big) \big(3-a-b\big)^4/4 + \big(3-a-b\big)^3 \big( 1 + 2n - a n \big).
\end{split}
\end{equation*}

Then
$
\beta'_{(X, D)}(D_2) 
= \big(1-2n\big)\big(3 - a - b\big)^3\big( 1 - 3b + a \big)/4 + \big(3 - a - b\big)^2\big( 1 + 2n - a n \big) \big(a-2b\big).
$
Compute $\beta'_{(X, \Delta')}(E)$ where $E$ is a unique prime divisor such that $E\sim H-F$. We have $L\sim (3-a-b)H+(1-n+bn)F = (3-a-b)(H - F) + ( 4 - a - n+b(n-1) )F$. Then
\begin{equation*}
\begin{split}
S'_{(X, D)}(E)
=& \int_0^{3-a-b}{\bigg(\big(3 - a - b - x\big) H + \big( 4 - a - n+b(n-1) \big)F \bigg)^3dx} \\
=& \big(n - 2\big) \big(3-a-b\big)^4/4 + \big(3-a-b\big)^3 \big( 4 - a - n+b(n-1) \big).
\end{split}
\end{equation*}

Then $
\beta'_{(X, \Delta')}(E) 
= (3-a-b)^3(n-2)(1+a+b) /4 + (3-a-b)^2( 4 - a - n+b(n-1)) ( a+~b).$
Note that $\beta'_{(X, D)}(D_1) + \beta'_{(X, D)}(D_2) + \beta'_{(X, D)}(E) = 0$. If not all the numbers $\beta'_{(X, D)}(D_1)$, $\beta'_{(X, D)}(D_2)$, $\beta'_{(X, D)}(E)$ are equal to zero, then one of them is negative, and we are done. Assume that $\beta'_{(X, D)}(D_1)=\beta'_{(X, D)}(D_2)=\beta'_{(X, D)}(E)=0$. It follows that $L^3=S'_{(X, D)}(D_2)/(1-b)=S'_{(X, D)}(E)$.
Then
\begin{multline*}
\big(1-2n\big) \big(3-a-b\big)^4/4 + \big(3-a-b\big)^3 \big( 1 + 2n - a n \big) \\
= \big(n - 2\big) \big(3-a-b\big)^4/4 + \big(3-a-b\big)^3 \big( 4 - a - n+b(n-1) \big) - bS'_{(X, D)}(E). 
\end{multline*}
Thus
$\big(n-1\big) \big(3-a-b\big)^4/4  =  - bS'_{(X, D)}(E)\leq 0.
$
Hence $n\leq1$ and $b=0$. However, this contradicts to the assumption that $D$ is reducible. Hence, we are done in this case.

\

\textbf{Case} D4. 
\hypertarget{compute-p2-bundle-cases-iv}
We show that in this case the log Fano pair $(X, D)$ is divisorially unstable for any values of the parameters. We have
\[
X = \mathbb{P}_{\mathbb{P}^1}\big(\oo\oplus\oo\oplus\oo(n)\big), \ \ \ \ \ n\geq 1, \ \ \ \ \ D_1\sim H + F, \ \ \ \ \ D_2\sim H - nF.
\]

Put $D = aD_1 + bD_2$, $0\leq a, b<1$. Then  
$
L=-K_X-D \sim \big(3 - a - b\big) H + \big( 2 - n - a +  b n \big)F =\big(3 - a - b\big) \big( H - n F \big) + \big( 2 + 2n - a - a n \big)F
$
is ample if and only if $a + n(1-b) < 2$. We have $A_{(X, D)}(D_2) = 1 - b$. 
It is easy to see that $L-xD_2$ is nef if and only if it is pseudo-effective if and only if $x\leq 3-a-b$.
Compute
\begin{equation*}
\begin{split}
S'_{(X, D)}(D_2)
=& \int_0^{3-a-b}{\bigg( \big(3 - a - b - x\big) \big( H - n F \big) + \big( 2 + 2n - a - a n \big)F \bigg)^3dx} \\
=& - n \big(3-a-b\big)^4/2 + \big(3-a-b\big)^3 \big( 2 + 2n - a - a n \big).
\end{split}
\end{equation*}

We have
\begin{multline*}
\beta'_{(X, D)}(D_2) = \big(1- b\big)\big(3 - a - b\big)^3\big(-2n\big) + 3\big(1-b\big)\big(3 - a - b\big)^2\big( 2 + 2n - a - a n \big) \\
+ n \big(3-a-b\big)^4/2 - \big(3-a-b\big)^3 \big( 2 + 2n - a - a n \big) \\
= \frac{1}{2} \Big(3 - a - b\Big)^2\bigg( a^2 (-n) - 2 a^2 + 2 a b n + 4 a b + 2 a n + 4 a - 3 b^2 n + 2 b n - 8 b - 3 n \bigg).
\end{multline*}

Using $n\geq 1$, estimate
\begin{multline*}
\frac{2\beta'_{(X, D)}(D_2)}{(3-a-b)^2} =
 n \Big(-\big(a-b\big)^2 -2\big(1-a\big) - b^2 - \big(1-b\big)^2\Big) - 2 a^2 + 4 a b + 4 a - 8 b  \\
\leq -6b(1-a)-3(1-a)^2-3b^2<0.
\end{multline*}

Hence, we are done in this case.


\

\textbf{Case} D5. 
\hypertarget{compute-p2-bundle-cases-v}
We show that in this case the log Fano pair $(X, D)$ is K-semistable for some values of the parameters. We have 
\[
X = \mathbb{P}^2 \times \mathbb{P}^1, \ \ \ \ \ \ D_1\sim H+F, \ \ \ \ \ \ D_2\sim H.
\]

Put $D = aD_1 + bD_2$ for $0\leq a, b<1$. Hence 
$
L=-K_X - D \sim \big(3-a-b\big)H + \big(2-a\big)F =  \big(3-a-b\big)(H+F) - \big(1-b\big)F
$
is ample. Then $L-xD_1 \sim \big(3-a-b-x\big)H + \big(2-a-x\big)F$ is nef if and only if it is pseudo-effective if and only if $x\leq 2-a$. We compute $\beta'_{(X, D)}(D_1)$. We have $A_{(X, D)}(D_1) = 1-a$. 
Then 
\begin{multline*}
\beta'_{(X, D)}(D_1) = 3\big(1-a\big)\big(3-a-b\big)^2\big(2-a\big) - \int_{0}^{2-a}{\bigg(\big(3-a-b-x\big)H + \big(2-a-x\big)F\bigg)^3 dx} \\
= \bigg(\big(-2+a\big) \Big(9 a^3+8 \big(-5+2 b\big)+2 a^2 \big(-29+8 b\big)+2 a \big(53-26 b+3 b^2\big)\Big)\bigg)/4.
\end{multline*}
which is negative for $2/3<a<1$, $0\leq b<1$. Now, we compute $\beta'_{(X, D)}(D_2)$:
\begin{multline*}
\beta'_{(X, D)}(D_2)= 3\big(1-b\big)\big(3-a-b\big)^2\big(2-a\big) - \int_{0}^{3-a-b}{\bigg(\big(3-a-b-x\big)H + \big(2-a\big)F\bigg)^3 dx} \\
=\big(3-a-b\big)^2\big(2-a\big)(a-2b). 
\end{multline*}

By Lemma 9.1 and Remark 9.6 in \cite{Fu16a} (they are formulated for the anticanonical polarisation, but work in our setting as well) it is enough to compute $\beta'_{(X, D)}(D')$ for $D'$ equivalent to one of the following divisors: $H, F, H+F, 2H+F$. However, $\beta'_{(X,D)}(H)\geq \beta'_{(X,D)}(D_2)$, $\beta'_{(X,D)}(H+F)\geq \beta'_{(X,D)}(D_1)$, and $\beta'_{(X, D)}(F)\leq \beta'_{(X, D)}(2H+F)$ so it is enough to compute $\beta'_{(X, D)}(F)$:
\begin{multline*}
\beta'_{(X, D)}(F)=3 \big(3-a-b\big)^2\big(2-a\big) - \int_{0}^{2-a}{\bigg(\big(3-a-b\big)H + \big(2-a-x\big)F\bigg)^3 dx} \\
=3 a \big(3-a-b\big)^2\big(2-a\big)/2>0.
\end{multline*}

We conclude that the pair $(X, D)$ is divisorially semistable if one of the following conditions holds:

$0\leq a\leq \alpha\approx0.507$, and $b\leq a/2$, where $\alpha$ is a root of the equation $-40+106\alpha-58\alpha^2+9\alpha^3=~0$, 

or $\alpha\approx0.507\leq a\leq 20/37$, and  

$(-4+13 a-4 a^2)/(3 a)-\sqrt{(32-88 a+84 a^2-34 a^3+5 a^4)/a^2}/(3 \sqrt{2})\leq b\leq a/2$.

\begin{lemma}
\label{lem-aut-product}
Let $G$ be the automorphism group of the pair $(X, \Delta)$. Then $G\simeq\mathrm{GL}(2, \mathbb{C})$. In particular, $G$ is reductive. 
\end{lemma}
\begin{proof}
We have $X= \mathbb{P}^2\times \mathbb{P}^1$. It is easy to see that $\mathrm{Aut}(X)\simeq\mathrm{PGL}(3, \mathbb{C})\times\mathrm{PGL}(2, \mathbb{C})$. Denote by $l=\pi_1(D_1)$ the image of $D_1$ via the natural projection $\pi_1\colon X\to \mathbb{P}^2$. Also, let $P$ be the image of a unique $(-1)$-curve on $D_2\simeq\mathbb{F}_1$ via $\pi_1$. 

Note that $D_1\simeq \mathbb{P}^1\times  \mathbb{P}^1$, and $C=D_2|_{D_1}$ is a curve of bidegree $(1, 1)$ on it. Using the curve $C$, we can identify the dual to the first copy of $\mathbb{P}^1$ with the second copy of $\mathbb{P}^1$ in $D_1\simeq \mathbb{P}^1\times  \mathbb{P}^1$. Note that $\mathrm{Aut}^0(X)\simeq\mathrm{PGL}(2, \mathbb{C})\times\mathrm{PGL}(2, \mathbb{C})$. Since $G$ preserves $C$, any automorphism of the first copy of $\mathbb{P}^1$ uniquely defines an automorphism of the second copy of $\mathbb{P}^1$.

Thus $G$ is contained in $\mathrm{Aut}(\mathbb{P}^2; l, P)$, that is, the subgroup in $\mathrm{Aut}(\mathbb{P}^2)\simeq \mathrm{PGL}(3, \mathbb{C})$ that preserves both $l$ and $P$. On the other hand, any element of this group can be lifted to $\mathrm{Aut}(X)$ in a unique way, and clearly it preserves $\Delta=D_1+D_2$. The fact that $\mathrm{Aut}(\mathbb{P}^2; l, P)$ is isomorphic to $\mathrm{GL}(2, \mathbb{C})$ is straightforward.
\end{proof}
As a consequence, there are the following $G$-invariant centers for the pair $(X, D)$: $D_1$, $D_2$, $Z=D_1\cap D_2$ and $Z'$ which is the $(-1)$-curve on $D_1\simeq\mathbb{F}_1$. To check K-polystability of the pair $(X, D)$, we apply Abban-Zhuang theory as in Proposition \ref{prop-formulas-for-beta} to $Z$ and $Z'$. 

Let $Y=D_2$. Let $Z$ be the intersection $D_1\cap D_2$ which is a curve of bidegree $(1,1)$ on $D_2$. We have $A_{(X, D)}(Y)/S_{(X, D)}(Y) \geq 1$ for such $a, b$ that $\beta'_{(X, D)}(D_1)> 0$. By Proposition \ref{prop-formulas-for-beta}, we have
\begin{multline*}
S(W^Y_{\bullet, \bullet}; Z) 
= \frac{3}{L^3} \int_0^{\infty}\int_0^{\infty}{\mathrm{vol}\Big((L-uY)|_Y - v Z\Big)dv du} \\
= \frac{3}{L^3} \int_0^{2-a}\int_0^{3-a-b-u}{\big(3-a-b-u-v,2-a-u)\big)^2dv du} \\
= \frac{3}{L^3} \int_0^{2-a}\int_0^{3-a-b-u}{2\big(3-a-b-u-v\big)\big(2-a-u\big)dv du} \\
= \big(2-a\big)^2 \big(34+3 a^2-28 b+6 b^2+4 a (-5+2 b)\big)/12\big(3-a-b\big)^2\big(2-a\big).
\end{multline*}

One checks that $
\frac{A_{(Y, D_Y)}(Z)}{S(W^Y_{\bullet, \bullet}; Z)}>1
$ for $a$ and $b$ such that the pair $(X, D)$ is divisorially semistable. Now we apply Abban-Zhuang theory to the $G$-invariant center $Z'$ which is the $(-1)$-curve on $D_1\simeq\mathbb{F}_1$. Put $Y'=D_1$. Then 
\begin{multline*}
S(W^{Y'}_{\bullet, \bullet}; Z') 
= \frac{3}{L^3} \int_0^{\infty}\int_0^{\infty}{\mathrm{vol}\Big((L-u{Y'})|_{Y'} - v Z'\Big)dv du} \\
= \frac{3}{L^3} \int_0^{\infty}\int_0^{\infty}{\mathrm{vol}\Big(\big(3-a-b-u\big)H|_{Y'}+(2-a-u)F|_{Y'} - v Z'\Big)dv du} \\
= \frac{3}{L^3} \int_0^{2-a}\int_0^{3-a-b-u}{\Big(\big(3-a-b-u-v\big)h+(2-a-u+v)f\Big)^2dv du} \\
= \frac{3}{L^3} \int_0^{2-a}\int_0^{3-a-b-u}{\Big(\big(3-a-b-u-v\big)^2+2\big(3-a-b-u-v\big)(2-a-u+v)\Big)dv du} \\
+ \frac{3}{L^3} \int_{2-a}^{5/2-a-b/2}\int_0^{u-(2-a)}{\big(5-2a-2u-b\big)^2dv du} \\
= (1/48 (-1+b)^4 +1/12\big(-2+a\big) \big(5 a^3+2 a^2 (-23+8 b)+2 a (71+b (-50+9 b))\\
+4 (-37+b (40+b (-15+2 b)))))/\big(3-a-b\big)^2\big(2-a\big).
\end{multline*}

One checks that $
\frac{A_{(Y', D_{Y'})}(Z')}{S(W^{Y'}_{\bullet, \bullet}; Z')}>1
$ for the values $a$ and $b$ such that the pair $(X, D)$ is divisorially semistable. Hence by Propositions \ref{prop-formulas-for-beta} and \ref{thm-G-invariant}, divisorial stability implies K-polystability, and we are done in this case.

\

\textbf{Case} D7. 
\hypertarget{compute-p2-bundle-cases-vii}
We show that in this case the log Fano pair $(X, D)$ is divisorially unstable for any values of the parameters. We have 
\[
X = \mathbb{P}^1\times \mathbb{P}^2, \ \ \ \ \ \ D_1\sim 2H, \ \ \ \ \ \ D_2\sim F.
\] 
Put $D = aD_1 + bD_2$. Hence $L=-K_X - D \sim (3-2a)H + (2-b)F$ is ample. We note that $(X, D)$ has product type, that is, $X=Y\times Z$, $D = D_Y\boxtimes D_Z$ where $(Y, D_Y) = (\mathbb{P}^2, D_Y=aQ)$ and $Q$ is a smooth conic, and $(Z, D_Z=bP)$ where $P$ is a point. By \cite{Zhu20b} the pair $(X, D)$ is K-semistable (resp., K-polystable) if and only if both $(Y, D_Y)$ and $(Z, D_Z)$ are so. On the other hand, in \cite{Fu20a} it is proven that $(\mathbb{P}^2, aQ)$ is K-semistable for $a\leq 3/4$ and K-polystable for $a<3/4$, see Theorem~\ref{thm-2-dim}. Also, one easily checks that the pair $(Z, D_Z=bP)$ is K-unstable for $b>0$. Thus, the pair $(X, D)$ is K-semistable if and only if $a\leq 3/4, b=0$, and K-polystable if and only if $a<3/4, b=0$. However, if $b=0$ then $D$ is irreducible, which contradicts our assumption. Hence, we are done in this case.

\

\textbf{Case} D8. 
\hypertarget{compute-p2-bundle-cases-viii}
We show that in this case the log Fano pair $(X, D)$ is divisorially unstable for any values of the parameters. We start with the special case $k=n=0$.
We have: 
\[
X=\mathbb{P}^1\times\mathbb{P}^2, \ \ \ \ \ \ \Delta=D_1+D_2+D_3, \ \ \ \ \ \ D_1\sim H, \ \ \ \ \ \ D_2\sim H, \ \ \ \ \ \ D_3\sim F.
\] 
Put $D = aD_1 + bD_2 + cD_3$ for $0\leq a, b, c<1$ and $(a, b, c)\neq (0, 0, 0)$. Then 
$
L = -K_X - aD_1 - b D_2 - c D_3 \sim (3 - a - b)H + (2-c)F
$ is ample. Similarly to the case \hyperlink{compute-p2-bundle-cases-vii}{D7}, we note that this pair is of product type. Hence, using \cite{Fu20a} we conclude that it is K-unstable. Now, consider the general case. We have
\[
X=\mathbb{P}_{\mathbb{P}^1}\big(\oo\oplus \oo(k)\oplus \oo(n)\big), \ \ \ \ \ n\geq 1, \ \ \ \ \ D_1\sim H-kF, \ \ \ \ \ D_2\sim H-nF, \ \ \ \ \ D_3\sim F.
\] 
Put $D = aD_1 + bD_2 + cD_3$ for $0\leq a, b, c <1$. Hence 
$
L=-K_X - D = \big(3-a-b\big)H + \big(2-\big(1-a\big)k-\big(1-b\big)n - c\big)F
$
is ample if and only if $\big(1-a\big)k+\big(1-b\big)n+c<2$. 
Note that $L-xD_2$
is nef if and only if it is pseudo-effective if and only if $x\leq 3-a-b$. 
Then
\begin{equation*}
\begin{split}
S'_{(X, D)}(D_2) =&\int_{0}^{3-a-b}{\bigg(\big(3-a-b-x\big)\big(H-nF\big) + \big(2-\big(1-a\big)k+n\big(2-a\big)-c\big)F\bigg)^3 dx} \\
=& \big(a+b-3\big)^4\big(k-2n\big)/4 - \big(a+b-3\big)^3\big(2-\big(1-a\big)k+n\big(2-a\big)-c\big).
\end{split}
\end{equation*}

We obtain
$
\beta'_{(X, D)}(D_2) 
= \frac{1}{4}\big(3-a-b\big)^2 \Big( \big(3-a-b\big)\big(k-2n\big) \big(1+a-3b\big) + 4\big(2-\big(1-a\big)k+n\big(2-a\big)-c\big) \big(a-2b\big)\Big).
$
Compute
\begin{equation*}
\begin{split}
S'_{(X, D)}(D_1)=&\int_{0}^{3-a-b}{\bigg(\big(3-a-b-x\big)\big(H-kF\big) + \big(2+\big(2-b\big)k-n\big(1-b\big)-c\big)F\bigg)^3 dx} \\
=&\big(a+b-3\big)^4\big(n-2k\big)/4 - \big(a+b-3\big)^3\big(2+\big(2-b\big)k-n\big(1-b\big)-c\big).
\end{split}
\end{equation*}

Then
$
\beta'_{(X, D)}(D_1)
=\big(3-a-b\big)^3\big(n-2k\big) \big(1-3a+b\big)/4 + \big(3-a-b\big)^2\big(2+\big(2-b\big)k-n\big(1-b\big)-c\big) \big(b-2a\big).
$
We estimate the following expression using $\big(1-a\big)k+\big(1-b\big)n+c<2$, $k\geq 0$, $n\geq 1$ and $b(1-b)\leq 1/2$:
\begin{multline*}
\frac{4(\beta'_{(X, D)}(D_1)+\beta'_{(X, D)}(D_2))}{(3-a-b)^2}
= k \Big( - \big(1 - 3 a - b\big)^2 - 4\big(2 - b - b^2\big) \Big)/3 \\+ 
n \Big( \big(1 + a - b \big)^2 
- 4 \big(1 - b + b^2\big) \Big) - 8 \big(a+b\big) + 4c \big( a + b \big) \\
\leq  
n \Big( \big(1-b\big)^2+2a \big( 1-b \big)+a^2 - 4 \big(1 - b + b^2\big) \Big) - 8 \big(a+b\big) + 4c \big( a + b \big) \\
< 
\Big(2-\big(1-a\big)k - c \Big)\big(1-b+2a\big) + n \Big(a^2 - 4+ 4b\big(1-b\big)\Big) - 8\big(a+b\big) + 4c \big( a + b \big) \\
\leq  
2\big(1-b+2a\big) - \Big(\big(1-a\big)k - c \Big) \big(1-b+2a\big) + n \big(a^2 - 2\big) - 4\big(a+b\big) + \big(4c-4\big) \big( a + b \big) \\
\leq 
b\big(-10+5c\big) - c +a\big(a - 1\big) + a\big(2c-3\big)<0. 
\end{multline*}

Hence, we are done in this case.

\section{Quadric bundles over $\mathbb{P}^1$}
\label{section-quadric-bundles}
\label{quadric_section}
In this section, we consider log smooth log Fano pairs $(X, \Delta)$ which have an extremal contraction of fiber type, whose generic fiber is a smooth quadric, as in \cite[9.2]{M83}. Since the contraction is extremal, $\Delta$ is proportional to $K_X$ over the base, so $\Delta$ has bidigree $(1,1)$ on the generic fiber which is a smooth quadric. 
There exists an embedding of $X$ into 
\[
W=\mathbb{P}_{\mathbb{P}^1}(\oo\oplus \oo(k) \oplus \oo(n) \oplus \oo(m)), \ \ \ \ \  0\leq k \leq n \leq m.
\]

One has $X \sim 2H + dF$, $\Delta \sim H + eF$ for some integers $d$ and $e$ where $F$ is a fiber of the projection and $H$ is the tautological divisor on $W$. We prove the following 

\begin{proposition}
\label{prop-do-not-exist}
There exists a unique series of three-dimensional log smooth log Fano pair $(X, \Delta=D_1+D_2)$ with an extremal quadric bundle structure and reducible boundary: 
\[
X\subset W=\mathbb{P}_{\mathbb{P}^1}(\oo\oplus\oo\oplus\oo\oplus\oo(m)), \ \ \ \ X\sim 2H, \ \ \ \ D_1\sim H-mF, \ \ \ \  D_2\sim F, \ \ \ \ m\geq 1.
\]
The corresponding log Fano pair of Maeda type $(X, D)$ is K-semistable for $m=1$ and some values of parameters $a$ and $b$. 
\end{proposition}
 
To prove the proposition, we reproduce the classification given in \cite[9.2]{M83}. However, we show that in \cite{M83} some cases are missing and some cases in fact do not occur. First, in Proposition \ref{prop-quadric-bundles-irreducible-classification} we obtain the classification assuming that $\Delta$ is irreducible. After that, in Proposition \ref{prop-unique-quadric-bundle} we consider the case of reducible $\Delta$.


\subsection{Irreducible boundary} 

\begin{proposition}
\label{prop-quadric-bundles-irreducible-classification}
Let $(X, \Delta)$ be a log smooth log Fano pair with $\dim X = 3$ and irreducible~$\Delta$. Assume that $X$ admits an extremal quadric bundle structure. Then there are the following possibilities for $(X, \Delta)$. 
\begin{enumerate}
\item
$X\subset W = \mathbb{P}_{\mathbb{P}^1}(\oo\oplus\oo\oplus\oo\oplus\oo(m)), \quad X\sim 2H+2F, \quad \Delta\sim H-mF, \quad m\geq1.$

\item
$X\subset W = \mathbb{P}_{\mathbb{P}^1}(\oo\oplus\oo\oplus\oo(1)\oplus\oo(m)), \quad X\sim 2H+F, \quad \Delta\sim H-mF, \quad m\geq1.$

\item
$X\subset W = \mathbb{P}_{\mathbb{P}^1}(\oo\oplus\oo\oplus\oo\oplus\oo(m)), \quad X\sim 2H+F, \quad \Delta\sim H-mF, \quad m\geq0.$

If $m=0$, then $X$ is a Fano threefold $2.25$.

\item
$X\subset W = \mathbb{P}_{\mathbb{P}^1}(\oo\oplus\oo(1)\oplus\oo(1)\oplus\oo(m)), \quad X\sim 2H+F, \quad \Delta\sim H-mF, \quad m\geq1.$

\item
$X\subset W = \mathbb{P}_{\mathbb{P}^1}(\oo\oplus\oo(1)\oplus\oo(1)\oplus\oo(m)), \quad X\sim 2H, \quad \Delta\sim H-mF, \quad m\geq1.$

If $m=1$, then $X$ is a Fano threefold $2.18$.

\item
$X\subset W = \mathbb{P}_{\mathbb{P}^1}(\oo\oplus\oo\oplus\oo(1)\oplus\oo(m)), \quad X\sim 2H, \quad \Delta\sim H-mF, \quad m\geq1.$

\item
$X\subset W = \mathbb{P}_{\mathbb{P}^1}(\oo\oplus\oo\oplus\oo\oplus\oo(m)), \quad X\sim 2H, \quad \Delta\sim H-mF, \quad m\geq1.$

\item
$X\subset W = \mathbb{P}_{\mathbb{P}^1}(\oo\oplus\oo\oplus\oo\oplus\oo(1)), \quad X\sim 2H, \quad \Delta\sim H, \quad m\geq1.$

If $m=1$, then $X$ is a Fano threefold $2.29$.


\item
$X \subset W = \mathbb{P}_{\mathbb{P}^1}(\oo\oplus \oo(1) \oplus \oo(1) \oplus \oo(m)), \quad X \sim 2H - F, \quad \Delta \sim H-mF, \quad m\geq 1$.

If $m=1$, then $X$ is a weak Fano threefold $2.3.2$.

\item
$X \subset W = \mathbb{P}_{\mathbb{P}^1}(\oo\oplus \oo(1) \oplus \oo(2) \oplus \oo(m)), \quad X \sim 2H - 2F, \quad \Delta \sim H-mF, \quad m\geq 2$,

\item
$X \subset W = \mathbb{P}_{\mathbb{P}^1}(\oo\oplus \oo(1) \oplus \oo(1) \oplus \oo(2)), \quad X \sim 2H - 2F, \quad \Delta \sim H-F$.
\end{enumerate}

Numbering of Fano varieties is as in \cite[12.3]{IP99}, and numbering for a weak Fano variety is as in \cite[2.3]{Ta09}.
\end{proposition}

\begin{proof}

We consider several cases.

\

\textbf{Case $d\geq0$}.
Since $\Delta$ is effective, we have {$m+e\geq 0$}. Compute
\begin{equation}
\begin{split}
0<\ (-K_\Delta)^2 
=&\ ( H + ( 2 - k - n - m - d - e )F )^2 (2H+dF)(H+eF) \\
=&\ 8 -2 (k+n+m)-3d-2e.
\end{split}
\end{equation}

Hence
$
2k + 2n + 3d \leq 2 k + 2n + 2(m+e) + 3d \leq 7.
$
Assume that a curve $C=(H-kF)(H-nF)(H-mF)$ given by the exact sequence 
\begin{equation}
\label{curve-C}
0\to \oo(k)\oplus\oo(n)\oplus\oo(m) \to \oo \oplus \oo(k) \oplus \oo(n) \oplus \oo(m) \to \oo \to 0.
\end{equation}
is contained in $X$. Then by ampleness of $-K_X-\Delta$, we have
\begin{equation}
\begin{split}
\label{ineq-sum}
0<&\ C\cdot(-K_X-\Delta)
= C \cdot (2H + (2- k - n - m - d - e)F) 
= 2- k - n - m - d - e.
\end{split}
\end{equation}
Thus, $k + n + (m+e) + d \leq 1$. Otherwise, $X\cap S$ is a curve. Consider a surface $S=(H-nF)(H-mF)$ given by the exact sequence
\begin{equation}
\label{surface-S}
0\to \oo(n)\oplus\oo(m) \to \oo \oplus \oo \oplus \oo(n) \oplus \oo(m) \to \oo \oplus \oo \to 0.
\end{equation}
Consider the intersection $X\cap S$. It is non-empty: indeed, $S$ is a fiberwise line in $\mathbb{P}^3$, and $X$ is a fiberwise quadric. Also, $X\cap S$ is not the whole $S$, hence it is a curve. Thus, we have: 
\[
0<(-K_X-\Delta)\cdot S\cdot X 
= 4-d-2e-2m-2n,\\
\]
hence
\begin{equation}
\label{new-condition}
2n+d \leq 2n + 2(m+e) + d \leq 3. 
\end{equation}

Note that \eqref{new-condition} is a consequence of \eqref{ineq-sum}, so in any case \eqref{new-condition} holds. Thus, there are the following possibilities:
\begin{enumerate}
\item
$d=2$, $k=n=0$, $m+e=0$, $m\geq 1$. This case is realised.
\item
$d=1$, $k=0$, $n=1$, $m+e = 0$, $m\geq 1$. This case is realised. 

\item
$d=1$, $k=1$, $n=1$, $m+e = 0$, $m\geq 1$. This case is realised. 

\item
$d=1$, $k=0$, $n=0$, $m+e = 0$, $m\geq 0$. This case is realised. 

\item
$d=1$, $k=0$, $n=0$, $m+e = 1$, $m\geq1$. In this case $X$ contains a curve $C\subset X\cap (H-mF)$ as in \eqref{curve-C}, which contradicts \eqref{ineq-sum}. Hence, this case does not occur.

\item
$d=0$, $k=0$, $n=1$, $m+e = 0$, $m\geq 1$. This case is realised. 

\item
$d=0$, $k=1$, $n=1$, $m+e = 0$, $m\geq 1$. This case is realised. 

\item
$d=0$, $k=0$, $n=0$, $m+e = 0$, $m\geq 1$. This case is realised. 

\item
$d=0$, $k=0$, $n=0$, $m+e = 1$, $m\geq 1$. We have
If $m\geq 2$, then $\Delta$ is reducible, so we may assume that $m=1$. Then
\[
X\subset \mathbb{P}_{\mathbb{P}^1}(\oo\oplus\oo\oplus\oo\oplus\oo(1)), \quad X\sim 2H, \quad \Delta\sim H.
\]

In this case, $X$ is a Fano variety $2.29$. This case is realised.

\end{enumerate}

\textbf{Case $d<0$.}
Consider a curve $C\simeq \mathbb{P}^1\subset W$ as in \eqref{curve-C}.
Note that $C\subset X$ as $C\cdot X<0$. By ampleness of $-K_X-\Delta$ we have
\begin{equation}
\label{ineq-sum2}
\begin{split}
0<&\ C\cdot(-K_X-\Delta)
= C \cdot (2H + (2- k - n - m - d - e)F) 
= 2- k - n - m - d - e.
\end{split}
\end{equation}


The following sequence is exact: 
\[
N_{C/W} = \oo(-k)\oplus\oo(-n)\oplus\oo(-m)\longrightarrow N_{X/W}|_C = \oo(d) \longrightarrow 0.
\]
Thus, {$m+d\geq 0$}. Moreover, as above we have $m + e\geq 0$.

\

\textbf{Subcase $d<0$, $e\geq 0$}. 
From \eqref{ineq-sum2} it follows that
\[
k + n + (m+d) + e \leq 1,
\] 
and each summand in non-negative. We have the following possibilities:
\begin{enumerate}
\item
$m+d=0$, $k=0$, $n=0$, $0\leq e\leq1$. 
One checks that, since $m>0$, any element in the linear system $|2H-mF|$ contains a component linearly equivalent to $H-mF$. In particular, no element in this linear system is smooth. Hence this case does not occur. 

\

\item
$m+d=1$, $k=0$, $n=0$, $e=0$. 
Note that $d<0$ implies $m\geq 2$. Then similar argument as in the previous case implies a contradiction.

\

\item$m+d=0$, $k=0$, $n=1$, $e=0$.
Similar argument as in the previous case implies that $m= 1$ and $d= -1$.  We have
One checks that $X$ is Fano and $(-K_X)^3=48$, hence by the classification of Fano threefolds we have $X\simeq \mathbb{P}^1\times \mathbb{F}_1$, in particular $\rho(X)=3$. This contradicts to the assumption that the quadric bundle structure is extremal. Hence this case does not occur. 
\end{enumerate}
\

\textbf{Subcase $d<0$, $e<0$}. 
Consider cases:

\begin{enumerate}
\item
$k=n=0$. 
In this case, the linear system $|2H+dF|$ has base divisor which is equivalent to $H-mF$, hence no element of this linear system is smooth. Hence this case does not occur.

\

\item
$k=0$ and $n>0$. 
We have $S\subset \mathrm{Bs}|2F+dF|$, $S\subset \mathrm{Bs}\,|\Delta|$, where $S$ is a surface given by \eqref{surface-S}.
However, this is impossible, since in this case $X\cap \Delta$ contains a fiberwise line, and $X\cap \Delta$ is a fiberwise conic, so it is reducible in every fiber of $W\to \mathbb{P}^1$, so $X\cap \Delta$ is not smooth. Hence this case does not occur.

\

\item
$k>0$. 
Consider a curve $C$ 
as in \eqref{curve-C}. Compute $C\cdot X = H (H-nF) (H-mF) (2H+dF) 
=2k+d$. Note that if $2k+d<0$, then the surface $S$ as in \eqref{surface-S} is contained in $X\cap \Delta$, which is a contradiction as in the previous case. Hence $2k+d\geq0$. Then we have 
\begin{equation}
\label{eq-to-add}
(2k+d) + (m+e) \leq (k + n + d) + (m + e) \leq 1. 
\end{equation}

We have $1\leq 8-2(k+n+m)-3d-2e\leq 8$, hence $-k-n-m-3d/2-e\leq 0$. Combining this with \eqref{eq-to-add}, we get $d\geq -2$. So either $d=-1$ or $d=-2$. Consider cases:
\begin{enumerate}
\item
$d=-1$, $k+n+d=1$, $k=n=1$, $m+e=0$. This case is realised. 
 
\item
$d=-2$, $k+n+d=0$, $k=n=1$, $m+e=1$.

In this case, if $m=1$, then $X$ is singular along the curve $C$ as in \eqref{curve-C}. If $m\geq 3$, then $\Delta$ is reducible. Then these cases do not occur. If $m=2$, one checks that $\rho(X)=3$, which contradicts our assumption. Hence, this case does not occur.   

\item
$d=-2$, $k+n+d=0$, $k=n=1$, $m+e=0$. 
One checks that in this case $\Delta$ is singular. Hence this case does not occur.

\item
$d=-2$, $k+n+d=1$, $k=1$, $n=2$ $m+e=0$. This case is realised. 
\end{enumerate}
\end{enumerate}
We obtained all the cases as in the statement of the proposition, so the proof is complete.
\end{proof}

\subsection{Reducible boundary} 
Now we consider the case of irreducible boundary $\Delta$, and prove the following proposition.

\begin{proposition}
\label{prop-unique-quadric-bundle}
There exists a unique log smooth log Fano pair $(X, \Delta)$ with $\dim X=3$ and reducible $\Delta$, such that it admits an extremal quadric bundle: 
\end{proposition}

Assume that $(X, \Delta)$ is a log smooth log Fano pair of dimension $3$, $X$ admits an extremal quadric bundle, and $\Delta$ is reducible. Then it is easy to see that $\Delta=D_1+D_2$. We may assume that $D_1$ is ``horizontal'', and $D_2$ is a fiber of a quadric bundle. It follows that $-K_X - D_1 = -K_X-\Delta+D_2$, and since $D_2$ is nef, we have that $(X, D_1)$ is a among the cases listed in Proposition \ref{prop-quadric-bundles-irreducible-classification}. Examining these possibilities cases by case, we see that the only case that occurs is 
\[
X\subset W=\mathbb{P}_{\mathbb{P}^1}(\oo\oplus\oo\oplus\oo\oplus\oo(m)), \ \ \ \ X\sim 2H, \ \ \ \ D_1\sim H-mF, \ \ \ \  D_2\sim F, \ \ \ \ m\geq 1.
\]

We call this series of cases \hyperlink{compute-quadric-bundle-case-1}{Q1}.


\begin{proposition}
\hypertarget{compute-quadric-bundle-case-1}
Let $(X, D)$ be a log Fano pair of Maeda type, such that the corresponding pair $(X, \Delta=D_1+D_2)$ is of type \hyperlink{compute-quadric-bundle-case-1}{Q1}: 
\[
X\subset W=\mathbb{P}_{\mathbb{P}^1}(\oo\oplus\oo\oplus\oo\oplus\oo(m)), \ \ \ \ X\sim 2H, \ \ \ \ D_1\sim H-mF, \ \ \ \  D_2\sim F, \ \ \ \ m\geq 1.
\]
Then the pair $(X, D=aD_1+bD_2)$ is divisorially unstable for $m\geq 2$, and for $m=1$ the pair is K-semistable for $0\leq a\leq(\sqrt{10}-2)/3$ and $0\leq b\leq 1-\sqrt{4+4a+3a^2}/\sqrt{6}$, and K-polystable if the upper bounds are strict.
\end{proposition}
\begin{proof}
Note that the nef cone of $X$ is generated by (restrictions of) $F$ and $H$, and these divisors are primitive in the Picard group. The cone of pseudo-effective divisors is generated by $F$ and $E\sim H-mF$, where $E$ is the exceptional divisor of the divisorial contraction on $X$ induced by the divisorial contraction on $W$. We have $K_W \sim -4H-(2-m)F$, $K_X \sim -2H-F = -2E-(2m+1)F$, hence $-K_X-\Delta \sim H+mF$ is ample. 

For $m=1$, we have the case 9.2(7) in \cite{M83}. One checks that in this case $X$ is a blow-up of a conic in a smooth quadric hypersurface in $\mathbb{P}^2$. In other words, $X$ is a Fano threefold $2.29$ as in \cite[12.3]{IP99}. Note that, in Maeda's notation, the cases $e=0$ and $e=1$ occur, while the case $e=2$ does not occur. We have $F^2=0, E^3 
=-4m, E^2F=
2$, 
$L = -K_X-D \sim 
(2-a)E+(2m+1-b)F$. Then 
$
L^3 = (2-a)^3E^3 + 3(2-a)^2(2m+1-b)E^2F = -4m(2-a)^3 + 6(2-a)^2(2m+1-b).
$
Compute $\beta_1'(D_1)$. We have $L-xD_1 = (2-a-x)E+(2m+1-b)F$. Then
\begin{multline*}
S'_{(X, D)}(D_1) 
= \int_0^{2-a}{ \Big( -4m\big(2-a-x\big)^3+6\big(2-a-x\big)^2\big(2m+1-b\big) \Big) dx} \\
= \big(2-a\big)^3 \big(2-2 b+\big(2+a\big) m\big).
\end{multline*}

Thus
$
\beta'_{(X, D)}(D_1) 
=\big(2-a\big)^2 \big(2-4 a-2 b+4 a b-3 a^2 m\big).
$
Now, we compute $\beta'_{(X, D)}(H)$. We have 
\begin{multline*}
S'_{(X, D)}(H) 
= \int_0^{2-a}{ \Big( -4m\big(2-a-x\big)^3+6\big(2-a-x\big)^2\big(2m+1-b-mx\big) \Big) dx} \\
=\big(2-a\big)^3 \big(4-4 b+(2+3 a\big) m\big)/2.
\end{multline*}

Then $\beta'_{(X, D)}(H) 
= \big(2-a\big)^2 \big(4-4 b+4 m+3 a^2 m+a \big(4-4 b+4 m\big)\big)/2\geq0$. 
Now, we compute $\beta'_{(X, D)}(D_2)$. We have $L-xD_2 \sim (2m+1-b-x)/m\ H + (-1-ma+b+x)/m\ E.$ 
We split $S_{(X, D)}(D_2)$ in two parts:
\begin{multline*}
\int_0^{1+am-b}{\mathrm{vol}(L-xD_2)dx}
= \int_0^{1+am-b}{ \Big(\big(2-a\big)E+\big(2m+1-b-x\big)F\Big)^3 dx} \\
= \int_0^{1+am-b}{ \Big( -4m\big(2-a\big)^3+6\big(2-a\big)^2\big(2m+1-b-x\big) \Big) dx} \\
=\big(2-a\big)^2 \Big(3+3 b^2+4 \big(1+a\big) m+a \big(4+a\big) m^2-2 b \big(3+2 \big(1+a\big) m\big)\Big).
\end{multline*}

Then
\[
\int_{1+am-b}^{2m+1-b}{\mathrm{vol}(L-xD_2)dx} = \int_{1+a-b}^{3-b}{(2m+1-b-x)^3/m^3\ H^3dx} 
= (2-a)^4m^2/2.
\]

We obtain 
$
\beta'_{(X, D)}(D_2) 
= -(2-a)^2 (-6+12 b-6 b^2+(4+4 a+3 a^2) m^2)/2.
$
This expression is negative once $m\geq 2$, hence we may assume $m=1$. If both $\beta'_{(X, D)}(D_1)\geq 0$ and $\beta'_{(X, D)}(D_2)\geq 0$, we have $0\leq a\leq(\sqrt{10}-2)/3$ and $0\leq b\leq 1-\sqrt{4+4a+3a^2}/\sqrt{6}$.

By remark Lemma 9.1 and Remark 9.6 in \cite{Fu16a} (they are formulated for the anticanonical polarisation, but work in our setting as well) we see that for such values of parameters the pair $(X, D)$ is divisorially semistable, and if the inequalities are strict, then it is divisorially stable.

To check K-polystability, we apply Abban-Zhuang theory as in Proposition \ref{prop-formulas-for-beta}. Put $Y=E$. Let $Z$ be the intersection $D_1\cap D_2$ which is a curve of bidegree $(1,1)$ on $D_2\sim F$, and a curve of bidegree $(1, 0)$ on $D_1=E$.

Let $G$ be the automorphism group of the pair $(X, \Delta)$. Then one can check that $G\simeq \mathrm{PGL}(2, \mathbb{C})$, cf. \cite[Lemma 7.7]{PCS19}. Note that $Z$ is the only proper $G$-invariant center on $X$ different from $D_i$. We have $Y\simeq \mathbb{P}^1\times \mathbb{P}^1$, $Y|_Y = (-1, 2)$, $F|_Y = (0, 1)$, and $Z\sim (1, 1)$ on $Y$. Compute
\begin{multline*}
S(W^Y_{\bullet, \bullet}; Z) 
= \frac{3}{L^3} \int_0^{\infty}\int_0^{\infty}{\mathrm{vol}\Big((L-uY)|_Y - v Z\Big)dv du} \\
= \frac{3}{L^3} \int_0^{\infty}\int_0^{3-2u-b}{\mathrm{vol}\Big(\big(-2+a+u-v, 7-2a-b-2u-v)\Big)dv du} \\
= \frac{3}{L^3} \int_{2-a}^{3-a-b/3}\int_0^{-2+a+u}{2\big(-2+a+u-v\big)\big(7-2a-b-2u-v\big)dv du} \\
+\frac{3}{L^3} \int_{3-a-b/3}^{7/2-a-b/2}\int_0^{7-2a-b-2u}{2\big(-2+a+u-v\big)\big(7-2a-b-2u-v\big)dv du} \\
=5 (3-b)^4/1944+7 (3-b)^4/7776=(3-b)^4/288.
\end{multline*}

Thus, we have 
$
\delta_Z(X, D)\geq 
\frac{A_{(Y, D_Y)}(Z)}{S(W^Y_{\bullet, \bullet}; Z)}= 288(1-b)/(3-b)^4>1
$
for such values $a$ and $b$ that the pair $(X, D)$ is divisorially stable. 
Hence by Propositions \ref{prop-formulas-for-beta} and \ref{thm-G-invariant}, in our case divisorial stability implies K-polystability. Hence, we are done in this case, and Proposition \ref{prop-do-not-exist} is proven.
\end{proof}

\section{$\mathbb{P}^1$-bundles over a surface}
\label{section-p1-bundles}
Let $(X, \Delta)$ be a log Fano threefold such that $X$ admits a $\mathbb{P}^1$-bundle structure $\pi\colon X\to Z$. 
According to \cite{M83}, if $\Delta$ is reducible, then the possibilities are presented among the cases \hyperlink{compute-p1-bundle-cases-remark}{C1}--\hyperlink{compute-p1-bundle-cases-9}{C10} in Table $1$. The main goal in this section is to prove the following

\begin{proposition}
\label{prop-unstable-p1-bundles}
Let $(X, D)$ be a log Fano pair of Maeda type such that $X$ admits an extremal $\mathbb{P}^1$-bundle. Assume that $D$ is reducible. Then $(X, D)$ is divisorially unstable, unless $(X, D)$ is as in the case \hyperlink{compute-p1-bundle-cases-8}{C9}. 
\end{proposition}
\hypertarget{compute-p1-bundle-cases-remark}
Note that the case \hyperlink{compute-p1-bundle-cases-remark}{C1} can be treated as a special case of \hyperlink{compute-p1-bundle-cases-ii}{C2} if we put $b=c$. Also, we may assume that $k\geq 1$, since otherwise we get $X=\mathbb{P}^2\times\mathbb{P}^1$, and we have considered this case in \hyperlink{compute-p2-bundle-cases-viii}{D8} for $k=n=0$ and $b=0$. Note that the case \hyperlink{compute-p1-bundle-cases-iii}{C3} for $k\geq 0$ can be considered as a special case of \hyperlink{compute-p1-bundle-cases-ii}{C2} with the value of parameter $c$ equal to $0$. 
Note that the case \hyperlink{compute-p1-bundle-cases-viii}{C7} can be treated as a special case of \hyperlink{compute-p1-bundle-cases-vii}{C5} if we put $n=0$ with $b=c$. Note that the case \hyperlink{compute-p1-bundle-cases-vii}{C8} can be treated as a special case of \hyperlink{compute-p1-bundle-cases-vii}{C5} if we put $n=1$ with $b=c$.
Now, we consider the remaining cases.

\

\textbf{Case} C2. \hypertarget{compute-p1-bundle-cases-ii}
We show that in this case the log Fano pair $(X, D)$ is divisorially unstable for any values of the parameters. We have
\begin{equation*}
\label{compute-p1-bundle-cases-ii}
X = \mathbb{P}_{\mathbb{P}^2}\big(\oo\oplus\oo(k)\big), \ \ \ \ \ \ k\geq 1, \ \ \ \ \ \ D_1\sim H, \ \ \ \ \ \ D_2\sim D_3\sim F, 
\end{equation*}

It is easy to check that
\begin{equation*}
\begin{split}
H^3 &= \big(H-kF\big)^3 =  k^2, \ \ \ \ \ \ \ \ \big(H-kF\big)^2F = -k, \ \ \ \ \ \ \ \ \big(H-kF\big)F^2 = 1, \\
\mathrm{Nef}(X) =&\ \mathbb{R}_{\geq 0} \big[H\big] + \mathbb{R}_{\geq 0} \big[F\big], \ \ \ \ \ \ \ 
\overline{\mathrm{Eff}}(X) = \mathbb{R}_{\geq 0} \big[H-kF\big] + \mathbb{R}_{\geq 0} \big[F\big].
\end{split}
\end{equation*}

Put $\Delta = aD_1 + bD_2 + cD_3$ for $0\leq a, b, c < 1$. 
Note that
$
L \sim -K_X - D \sim \big(2-a\big)\big(H-kF\big) + \big(3+k-b-c\big)F 
= \big(2-a\big)H + \big(3-k+ak-b-c\big)F
$
is ample if and only if $k+b+c<3+ak$. 
We estimate $\beta'_{(X, \Delta)}(D_1)$. Then
$
L - x D_1 \sim \big(2-a-x\big)\big(H-kF\big) + \big(3+k-b-c\big)F 
= \big(1 + \big(3-b-c\big)/k\big) H - \big(a-1+ \big(3-b-c\big)/k+x\big) \big(H-kF\big)
$
is nef if and only if it is pseudo-effective if and only if $x\leq 2-a$. We have $A_{(X, D)}(D_1) = 1-a$. Compute
\begin{multline*}
S'_{(X, D)}(D_1)=\int_{0}^{2-a}{\bigg(\big(1 + \big(3-b-c\big)/k\big) H - \big(a-1+ \big(3-b-c\big)/k+x\big) \big(H-kF\big)\bigg)^3 dx} \\ 
= \big(1/k\big) \big(k + 3-b-c\big)^3 \big( 2 - a \big) - 
\big(1/k^2\big) \big(k + \big(3-b-c\big)\big)^4/4 \\ 
+ \big(1/k^2\big) \big(ak - k + \big(3-b-c\big)\big)^4/4.
\end{multline*}

Finally, we obtain
\begin{multline*}
k\beta'_{(X, D)}(D_1) = \big(1-a\big)\big(k + 3-b-c\big)^3 - \big(1-a\big)\big(ak - k + 3 - b - c \big)^3 \\
- \big(k + 3-b-c\big)^3 \big( 2 - a \big) + \big(1/k\big) \big(k + 3-b-c\big)^4/4 
- \big(1/k\big) \big(ak - k + 3-b-c\big)^4/4 \\
<\big(k + 3-b-c\big)^3 \big(-3/4+\big(3-b-c\big)/4k\big)\leq0. 
\end{multline*}
since $k\geq 1$ and by the above $k+b+c<3+ak$, so $ak - k + 3 - b - c>0$. Hence, we are done in this case.

\

\textbf{Case} C3. 
\hypertarget{compute-p1-bundle-cases-iii}
We show that in this case the log Fano pair $(X, D)$ is divisorially unstable for any values of the parameters. We only have to consider the case $k=-1$. We will write it is follows: 
\begin{equation*}
\begin{split}
X=\mathbb{P}_{\mathbb{P}^2}(\oo\oplus \oo(1)), \ \ \ \ \ \ D_1\sim H, \ \ \ \ \ \ \ D_2\sim F.
\end{split}
\end{equation*}

Put $\Delta = aD_1 + bD_2$ for $0\leq a, b< 1$. Then $K_X \sim -2H - 2F$. Note that
$
L \sim -K_X - D \sim \big(2-a\big)H + \big(2-b\big)F
$
is ample. We estimate $\beta'_{(X, \Delta)}(E)$ where $E$ is a unique divisor equivalent to $H-F$. Then
$
L - x E \sim \big(2-a-x\big)H + \big(2-b+x\big)F
$
is nef if and only if it is pseudo-effective if and only if $x\leq 2-a$.  Compute
\begin{multline*}
\beta'_{(X, D)}(E) = \big(2-a\big)^3 + 3\big(2-a\big)^2\big(2-b\big) + 3\big(2-a\big)\big(2-b\big)^2 \\
- \int_0^{2-a}{\bigg(\big(2-a-x\big)^3 + 3\big(2-a-x\big)^2\big(2-b+x\big) + 3\big(2-a\big)\big(2-b+x\big)^2\bigg)dx} \\
= \big(2-a\big) \bigg(19 a^3+8 \big(-43+13 b\big)+2 a^2 \big(-87+16 b\big)+2 a \big(234-70 b+3 b^2\big)\bigg)/4 \\
\leq  \big(2-a\big) \bigg(19 a^3-240-142 a^2+468 a \bigg)/4
\end{multline*}
which is negative for $a>3/5$. On the other hand, for $a\leq 3/5$ we estimate 
\begin{multline*}
\beta'_{(X, D)}(E) = \big(2-a\big)\big(19 a^3 + 32 a^2 b - 174 a^2 + 6 a b^2 - 140 a b + 468 a + 104 b - 344\big)  \\
\leq\big(2-a\big)\big(19 a^3 + 32 a^2 b + 6 a b^2 - 140 a b + 468 a(1-174a/468) + 104 b - 344\big)  \\
\leq\big(2-a\big)\big(19 a^3 + 32 a^2 b + 6 a b^2 - 140 a b + 294 + 104 b - 344\big) \\
\leq\big(2-a\big)\big(450 b^2 + 3940 b - 5737\big)/125
\end{multline*}
which is negative for $0\leq b<1$. Hence, we are done in this case.


\

\textbf{Case} C4. \label{compute-p1-bundle-cases-iv}
\hypertarget{compute-p1-bundle-cases-iv}
We show that in this case the log Fano pair $(X, D)$ is divisorially unstable for any values of the parameters. We have 
\[
X=\mathbb{P}_{\mathbb{F}_n}(\oo\oplus \oo(ks+(kn+m)f)), \ \ \ \ \ k, n\geq 0, \ \ \ \ \ m\geq -1, \ \ \ \ \ D_1\sim H_1, \ \ \ \ \ D_2\sim F_s,
\] 
where $H_1\sim H - k F_s - ( k n + m )F_f$ for the tautological divisor $H$, and $F_s=\pi^{-1}(s)$ is the preimage of the $(-n)$-curve $s$ via the structure morphism $\pi\colon X\to \mathbb{F}_{n}$. Note that $D_2\simeq \mathbb{F}_{m}$ for $m\geq 0$, $D_2\sim \mathbb{F}_1$ for $m=-1$. 
Also, we have $(D_1\cap D_2)_{D_2}^2=-m$, $(D_1\cap D_2)_{D_1}^2=-n$.

Note that the computation in \cite[8.3]{M83} is not correct: it is claimed that only the case $k=n=0$ occurs. More precisely, in the Case $(i)$ at the top of page $115$, we should have (in Maeda's notation) $-K_V-D_1\sim D_1+(2+k)V_{\Delta}+(n+kn)V_M$. Then the equality after the formula for $-K_V-D_1$ simplifies as $8=8$ which does not give us any information about the possible values of $k=k$ and $n=n$. We claim that for arbitrary $k, n\geq0, m\geq -1$ the pair $(X, \Delta=D_1+D_2)$ is log Fano. In the Lemma \ref{lem-log-Fano-criterion2} we demonstrate that the pair $(X, \Delta)$ is a log smooth log Fano pair for any $k, n\geq 0, m\geq-1$. We will deal with this case as a special case of \hyperlink{compute-p1-bundle-cases-vii}{C5} for $c=0$.

\begin{lemma}
\label{lem-log-Fano-criterion2}
Let $(X, \Delta)$ be a log smooth pair as in the case \hyperlink{compute-p1-bundle-cases-iv}{C4}. Then it is log Fano for $k, n\geq 0$, $m\geq-1$. 
\end{lemma}
\begin{proof}
We have
$
-K_X - \Delta = D_1+(1+k)D_2+(2+n+kn+m)V_M, \ \ \ \ 
D_1|_{D_1} = -ks-(m+kn)f.
$
It is easy to see that $\overline{\mathrm{NE}}(X)$ is generated by fibers and sections on the ruled surfaces $D_1$ and $D_2$. Hence it is enough to check that the restrictios of $-K_X-\Delta$ to $D_1$ and $D_2$ are ample. We have 
$
-K_X - \Delta|_{D_1} 
= s + (2+n)f
$
which is ample on $D_1\simeq \mathbb{F}_n$. Note that $D_1|_{D_2}=s$ for $m\geq 0$, and $D_1|_{D_2}\sim s+f$ for $m=-1$. 
Hence for $m\geq 0$ we have
$
-K_X - \Delta|_{D_2} = s + ((1+k)(-n)+2+n+kn+m)f 
= s + (2+m)f
$
which is ample. For $m=-1$, we have 
$
-K_X - \Delta|_{D_2} = s + 2f,
$
which is ample.
\end{proof}



\textbf{Case} C5. \label{compute-p1-bundle-cases-vii}
\hypertarget{compute-p1-bundle-cases-vii}
We show that in this case the log Fano pair $(X, D)$ is divisorially unstable for any values of the parameters. We have 
\[
X=\mathbb{P}_{\mathbb{F}_n}\big(\oo\oplus \oo(ks+(m+kn)f)\big), \ \ \ \ k, n\geq 0, \ \ \ \ m\geq -2, \ \ \ \ D_1\sim H_1, \ \ \ \ D_2\sim F_s, \ \ \ \ D_3 \sim F_f,
\] 
where $H_1\sim H - k F_s - ( k n + m )F_f$ for the tautological divisor $H$, and $F_s=\pi^{-1}(s)$ is the preimage of the $(-n)$-curve $s$ via the structure morphism $\pi\colon X\to \mathbb{F}_{n}$, and $F_f = \pi^{-1}(f)$ is the preimage of a fiber on 
$\mathbb{F}_{n}$. We have
$
K_X \sim -2H_1 - \big(2+k\big)F_s - \big(2+n+kn+m\big)F_f. 
$
Also, $D_1\simeq \mathbb{F}_{n}$, $D_3\simeq \mathbb{F}_{k}$, $D_2\simeq \mathbb{F}_{m}$ for $m\geq 0$, $F_s\simeq \mathbb{F}_1$ for $m=-1$ and $F_s\simeq \mathbb{F}_2$ for $m=-2$. 
We have ${H_1}|_{H_1} = - k s - \big(kn+m\big)f$. Compute the intersection theory:
\begin{equation*} 
\begin{split}
H_1^3 = n k^2 +& 2k m, \ \ \ \ \ \ \ \ H_1 F_s F_f = 1, \ \ \ \ \ \ \ \ F_f^2 = 0, \ \ \ \ \ \ \ \ F_s^2 H_1 = -n, \\
F_s^2 F_f =& 0, \ \ \ \ \ \ \ \ F_s^3 = 0, \ \ \ \ \ \ \ \ H_1^2 F_f = -k, \ \ \ \ \ \ \ \ H_1^2 F_s = -m.
\end{split}
\end{equation*}

Put $D = aD_1 + bD_2 + cD_3$. Then
$
L \sim -K_X - D = \big(2-a\big)H_1 + \big(2+k-b\big)F_s + \big(2+n+k n+m-c\big)F_f.
$
One checks that $\overline{\mathrm{Eff}}(X)=\mathbb{R}_{\geq0}\big[H_1\big] + \mathbb{R}_{\geq0}\big[F_s\big] + \mathbb{R}_{\geq0}\big[F_f\big]$ and 
\[
\mathrm{Nef}(X)=\mathbb{R}_{\geq0}\big[F_f\big] + \mathbb{R}_{\geq0}\big[F_s+n F_f\big] + \mathbb{R}_{\geq0}\big[H_1 + k F_s + ( k n + m )F_f\big].
\]
Then the pair $(X, D)$ is log Fano if and only if the following conditions hold:
$k(1-a) <2-b$, and $m(1-a) + n(1-b) < 2-c$.
Compute 
\begin{multline*}
L^3 
=\big(2-a\big)^3 \big(nk^2 + 2km\big) + 6 \big(2-a\big)\big(2+k-b\big)\big(2+n+kn+m-c\big) \\
+ 3\big(2+k-b\big)^2\big(2-a\big)\big(-n\big) + 3\big(2-a\big)^2\big(2+n+kn+m-c\big)\big(-k\big) \\
+ 3\big(2-a\big)^2\big(2+k-b\big)\big(-m\big).
\end{multline*}

We have $A_{(X, D)}(D_1)=1-a$. It is easy to check that the pseudo-effective threshold is $2-a$. Compute
\begin{multline*}
S'_{(X, D)}(D_1) 
= \int_{0}^{2-a}{\bigg(\big(2-a-x \big)H_1 + \big(2+k-b \big)F_s + \big(2+n+kn+m-c \big)F_f \bigg)^3dx} \\
= \big(a-2 \big)^4 \big(nk^2 + 2km \big)/4 + 3 \big(a-2 \big)^2\big(2+k-b\big) \big(2+n+kn+m-c \big) \\
- 3 n \big(a-2 \big)^2 \big(2+k-b \big)^2/2 + k \big(a-2 \big)^3 \big(2+n+kn+m-c \big) + m \big(a-2 \big)^3 \big(2+k-b \big).
\end{multline*}

Finally, we obtain
\begin{multline*}
\beta'_{(X, D)}(D_1)
= \big(1-a \big) \big(2-a \big)^3 \big(nk^2 + 2km \big) + 6 \big(1-a \big) \big(2-a \big) \big(2+k-b \big) \big(2+n+kn+m-c \big) \\
-3n \big(1-a \big) \big(2+k-b \big)^2 \big(2-a \big) - 3k \big(1-a \big) \big(2-a \big)^2 \big(2+n+kn+m-c \big) \\
- 3m \big(1-a \big) \big(2-a \big)^2 \big(2+k-b \big) - \big(2-a \big)^4 \big(nk^2 + 2km \big)/4 \\ - 3 \big(2-a \big)^2\big(2+k-b\big) \big(2+n+kn+m-c \big) 
+ 3 n \big(2-a \big)^2 \big(2+k-b \big)^2/2 \\ + k \big(2-a \big)^3 \big(2+n+kn+m-c \big) + m \big(2-a \big)^3 \big(2+k-b \big). 
\end{multline*}

This expression is negative provided that $k(1-a)<2-b$ and $m(1-a)+n(1-b)<2-c$. 
If $m=-1,-2$ similarly we check that $\beta'_{(X,D)}(D'_1)+\beta'_{(X,D)}(D_1)<0$ where $D'_1\sim H \sim H_1 + kF_S+(k n + m)F_f$. Hence, we are done in this case.

\

\textbf{Case C6}.  \hypertarget{compute-p1-bundle-cases-new}{}
We show that in this case the pair $(X, D)$ is divisorially unstable for any values of the parameters. We have $X=\mathbb{P}_{\mathbb{F}_n}(\ee)$ where $\ee$ is given by the following non-split exact sequence:
\[
0\to \oo \to \ee \to \oo(-ks-(kn-2)f) \to 0, \ \ \ \ \ k, n \geq 1, \ \ \ \ \ (k, n)\neq (1, 1).
\]
This exact sequence defines a section $H_1$ of the structure morphism $\pi\colon X\to \mathbb{F}_n$. Note that $H_1\simeq \mathbb{F}_n$. Denote by $F_s=\pi^{-1}(s)$ the preimage of the $(-n)$-section on $\mathbb{F}_n$. We have $F_s\simeq \mathbb{P}^1\times\mathbb{P}^1$. We check that $(X, \Delta=H_1+F_s)$ is a log smooth log Fano pair. Note that the Mori cone $\overline{\mathrm{NE}}(X)$ is generated by the curves $l, f, s$, where $l$ is a fiber of $\pi$, $f$ is a fiber on $H_1\simeq \mathbb{F}_n$, and $s$ is a section on $H_1$. 
We have
$-K_X - \Delta = H_1+(1+k)F_s+(n+kn)F_f$, and  
$D_1|_{D_1} = -ks-(kn-2)f.$
By the above, it is enough to check that the restrictions of $-K_X-\Delta$ to $D_1$ and $D_2$ are ample. We have 
$-K_X - \Delta|_{D_1} 
= s + (2+n)f$
which is ample on $D_1\simeq \mathbb{F}_n$. Note that $D_1|_{D_2}$ has bidegree $(1,1)$ on $D_2$. Hence we have
$-K_X - \Delta|_{D_2} \sim (1,1)$, which is ample. 

Put $(X, D)$ for $D=aD_1+bD_2$ where $0\leq a, b<1$. This case is similar to the previous one if we put $m=-2$ and $c=0$. Indeed, we have the same formulas for the canonical class and the intersection theory. Note that $\mathrm{Pseff}(X)$ is generated by $H_1, F_s, F_f$. Thus, the same formulas as in the previous case apply in this case, which show that the pair is divisorially unstable for any values of parameters.

\

\textbf{Case C7}. 
\hypertarget{compute-p1-bundle-cases-viii}
We show that in this case the log Fano pair $(X, D)$ is divisorially unstable for any values of the parameters. We have 
\[
X=\mathbb{P}_{\mathbb{P}^1\times \mathbb{P}^1}\big(\oo\oplus \oo(k, n)\big), \ \ \ \ \ k\geq 0, \ \ \ \ \ n\geq 0, \ \ \ \ \ D_1\sim H_1, \ \ \ \ \ D_2\sim F_1 + F_2,
\] 
where $H_1\sim H-kF_1-nF_2$ for the tautological divisor $H$.
Put $D = aD_1 + bD_2$ for $0\leq a,b<1$. Then $L=-K_X - D \sim \big(2-a\big)H_1 + \big(2+k-b\big)F_1 + \big(2+n-b\big)F_2$ is ample. 
Compute the intersection theory:
\begin{equation*}
\begin{split}
F_1^2=F_2^2=0, \ \ \ \ \ \ H_1F_1F_2 = 1, \ \ \ \ \ \ H_1^3 = 2kn, \ \ \ \ \ \ H_1^2 F_1 = -n, \ \ \ \ \ \ H_1^2 F_2 = -k,
\end{split}
\end{equation*}

We estimate $\beta'_{(X, D)}(D_1)$. Note that $L - xD_1$ 
is pseudo-effective if and only if $x\leq 2-a$. We have $A_{(X, D)}(D_1)=1-a$. Compute 
\begin{multline*}
S'_{(X, D)}(D_1)=\int_0^{2-a}{\bigg(\big(2-a-x\big)H_1 + \big(2+k-b\big)F_1 +\big(2+n-b\big)F_2\bigg)^3dx} \\
=kn\big(a-2\big)^4/2 + 3\big(a-2\big)^2 \big(2+k-b\big) \big(2+n-b\big) \\
+n \big(a-2\big)^3 \big(2+k-b\big) + k \big(a-2\big)^3 \big(2+n-b\big).
\end{multline*}

Finally, we get
\begin{multline*}
\beta'_{(X, D)}(D_1)
=kn\big(2-a\big)^3\big(-3a/2\big) + n\big(2-a\big)\big(2+k-b\big)\big(-2a^2+2a-2\big) \\
+ k\big(2-a\big)\big(2+n-b\big)\big(-2a^2+2a-2\big) \\
- 3a\big(2-a\big)\big(2+k-b\big)\big(2-b\big) - 3a\big(2-a\big)\big(2-b\big)\big(2+n-b\big)
\end{multline*}
which is negative once we have $a>0$ or $k>0$ or $n>0$. 

If $a=k=n=0$, we get the case $X=\mathbb{P}^1\times\mathbb{P}^1\times\mathbb{P}^1$, $D_1=H$, $D_2=F_1+F_2$. Put $D = aD_1 + bD_2$. Then 
$
L=-K_X - D = \big(2-a\big)H + \big(2-a\big)F_1 + \big(2-b\big)F_2
$
is ample. Note that $(X, D)$ is or product type: $X=Y\times Z$, $D = D_Y\boxtimes D_Z$ where $(Y, D_Y)=(\mathbb{P}^1\times \mathbb{P}^1, aL)$, $L$ is a curve of bidegree $(1, 1)$, and $(Z, D_Z)=(\mathbb{P}^1, bP)$ where $P$ is a point. Arguing as in the case \hyperlink{compute-p2-bundle-cases-vii}{D7} and using Theorem \ref{thm-2-dim}, we see that $(X, D)$ is K-semistable (resp., K-polystable) if and only if $a\leq 1/2$ (resp., $a<1/2$) and $b=0$. This implies that $D$ is irreducible, which contradicts our assumption. Hence we are done in this case.

\

\textbf{Case C9}
\hypertarget{compute-p1-bundle-cases-8}
We show that in this case the log Fano pair $(X, D)$ is divisorially semistable (and K-polystable) for some values of the parameters. Let $X=\mathbb{P}_{\mathbb{P}^2}(T_{\mathbb{P}^2})$. In other words, $X$ is a divisor of bidegree $(1, 1)$ in $\mathbb{P}^2\times\mathbb{P}^2$. Let $\Delta=D_1+D_2$ where $D_1$ has bidegree $(1, 0)$, and $D_2$ has bidegree $(0, 1)$. Note that $D_1\simeq D_2 \simeq \mathbb{F}_1$. Put $D=aD_1+bD_2$, where $0\leq a,b<1$. Then $L=-K_X-D\sim (2-a, 2-b)$ is ample. Then 
$
L^3 = 
3(2-a)^2(2-b)+3(2-a)(2-b)^2.
$
Compute 
\begin{multline*}
\beta'_{(X, D)}(D_1) = 3(1-a)(2-a)(2-b)(4-a-b) - \int_0^{2-a}{\mathrm{vol}(2-a-x,2-b)dx} \\
= 3(1-a)(2-a)(2-b)(4-a-b) - \int_0^{2-a}{3(2-a-x)^2(2-b)+3(2-a-x)(2-b)^2dx} \\
= 3(1-a)(2-a)(2-b)(4-a-b) - (2-a)^3(2-b) - 3(2-a)^2(2-b)^2/2.
\end{multline*}

The pair is divisorially semistable for $0<a\leq 2-\sqrt{3}$ and 
$0\leq b\leq(16-3 a)/8-\sqrt{3} \sqrt{64-32 a+3 a^2}/8$ 
or $2-\sqrt{3}\leq a\leq2/7$ and $(-4+16 a-4 a^2)/(3 a)\leq b\leq(16-3 a)/8- \sqrt{3} \sqrt{64-32 a+3 a^2}/8$.

To check K-polystability, we apply Abban-Zhuang theory as in Proposition \ref{prop-formulas-for-beta}. Put $Y=D_1$. Let $Z$ be the intersection $D_1\cap D_2$ which is equivalent to a positive section $h$ on $D_1\simeq D_2\simeq \mathbb{F}_1$. 

\begin{lemma}[{cf. \cite[Lemma 7.7]{PCS19}}]
\label{lem-aut-tangent-bundle}
Let $G$ be 
the automorphism group of the pair $(X, \Delta)$. Then $G\simeq\mathrm{GL}(2, \mathbb{C})\rtimes\mathbb{Z}/2$. In particular, $G$ is reductive. 
\end{lemma}
\begin{proof}
We have $X\subset \mathbb{P}^2_1\times \mathbb{P}^2_2$ where we index two copies of $\mathbb{P}^2$ for convenience. It is well known that the automorphism group of $X$ is isomorphic to $\mathrm{PGL}(3, \mathbb{C})\rtimes \mathbb{Z}/2$. One can check that each of the two natural projections $\pi_i\colon X\to\mathbb{P}^2_i$ for $i=1,2$ induces an isomorphism $\mathrm{Aut}^0(X)\simeq \mathrm{Aut}(\mathbb{P}^2)\simeq \mathrm{PGL}(3, \mathbb{C})$. Using this observation, we show that the group~$G$ is isomorphic to  $\mathrm{Aut}(\mathbb{P}^2; l, P)\rtimes \mathbb{Z}/2$ where $\mathrm{Aut}(\mathbb{P}^2; l, P)$ is the automorphism group of $\mathbb{P}^2$ that preserves a line $l$ and a point $P$. Let $l_1=\pi_1(D_1)\subset \mathbb{P}^2_1$ and  $l_2=\pi_2(D_2)\subset \mathbb{P}^2_2$ be the lines on two copies of $\mathbb{P}^2$. Using $X$, we can identify $(\mathbb{P}^2_1)^\vee$, that is, the dual to the first copy of the projective plane, with $\mathbb{P}^2_2$.  

A line $l_1\subset \mathbb{P}^2_1$ 
under the above identification corresponds to a point $P_2\in \mathbb{P}^2_2$ defined as follows. The preimage $D_1=\pi_1^{-1}(l_1)\subset X$ is isomorphic to a Hirzebruch surface $\mathbb{F}_1$. Then, one checks that $P_2$ is the image of a unique $(-1)$-curve on $D_1$ via the projection map $\pi_2$. 
Note that $P_2\not\in l_2$. Indeed, this follows from the fact that the surfaces $D_1$ and $D_2$ intersect in a smooth rational curve $h$ which is a positive section on both $D_1\simeq D_2\simeq\mathbb{F}_1$. Since the positive section on $\mathbb{F}_1$ is disjoint from the $(-1)$-curve, we have $P_2\not\in l_2$. 

Then, the connected component of the identity of $G$ is contained in $\mathrm{Aut}(\mathbb{P}^2_2; l_2, P_2)$. On the other hand, any element of this group can be lifted to $\mathrm{Aut}^0(X)\simeq\mathrm{PGL}(3, \mathbb{C})$ in a unique way, and clearly it preserves $\Delta=D_1+D_2$. Note that, up to change of coordinates, there is only one isomorphism class of pair $(X, D_1+D_2)$. Hence we may assume that $D_1$ is given by the equation $x_0=0$, and $D_2$ is given by the equation $y_0=0$. Thus, there exists an involution that changes $D_1$ and $D_2$. Hence, we have $G \simeq\mathrm{Aut}(\mathbb{P}^2_2; l_2, P_2)\rtimes\mathbb{Z}/2$. The fact that $\mathrm{Aut}(\mathbb{P}^2_2; l_2, P_2)$ is isomorphic to $\mathrm{GL}(2, \mathbb{C})$ is straightforward.

\end{proof}
By Proposition \ref{thm-G-invariant}, to check K-polystability we may consider only irreducible $G$-invariant centers for the pair $(X, D)$: $D_1$, $D_2$ and $Z=D_1\cap D_2$. 
We apply Abban-Zhuang theory as in Proposition \ref{prop-formulas-for-beta} to $Z$. 
Put $Y=D_1\simeq\mathbb{F}_1$ and $Z=D_1\cap D_2$, $Y|_Y \sim f$, and $Z\sim h$ on $Y$. We have
\begin{multline*}
S(W^Y_{\bullet, \bullet}; Z) 
= \frac{3}{L^3} \int_0^{\infty}\int_0^{\infty}{\mathrm{vol}\Big((L-uY)|_Y - v Z\Big)dv du} \\
= \frac{3}{L^3} \int_0^{\infty}\int_0^{\infty}{\mathrm{vol}\Big(\big(2-a-u\big)H_1|_{H_1} + \big(2-b\big)H_2|_{H_1} - v h\Big)dv du} \\
= \frac{3}{L^3} \int_0^{\infty}\int_0^{\infty}{\mathrm{vol}\Big(\big(2-b-v\big)s + \big(4-a-b-u-v\big)f\Big)dv du} \\
= \frac{3}{L^3} \int_0^{2-a}\int_0^{2-b}{\Big(2\big(2-a-u\big)\big(2-b-v\big) + \big(2-b-v\big)^2\Big)dv du} \\
+ \frac{3}{L^3} \int_{2-a}^{4-a-b}\int_0^{4-a-b-u}{\big(4-a-b-u-v\big)^2dv du} \\
= \frac{\big(1/6 (-2+a) (3 a+2 (-5+b)) (-2+b)^2 + 1/12 (-2+b)^4\big)}{(2-a)^2(2-b)+(2-a)(2-b)^2}.
\end{multline*}

One checks that  
$
\frac{A_{(Y, D_Y)}(Z)}{S(W^Y_{\bullet, \bullet}; Z)}>1
$
for such values $a$ and $b$ that the pair $(X, D)$ is divisorially stable.

\

\textbf{Case C10.}
\hypertarget{compute-p1-bundle-cases-9}
We show that in this case the log Fano pair $(X, D)$ is divisorially unstable for any values of the parameters. We have $X=\mathrm{Bl}_l Y$ where $Y$ is a divisor of bidegree $(1,1)$ in $\mathbb{P}^2\times \mathbb{P}^2$, and $l$ is a fiber of a $\mathbb{P}^1$-bundle $Y\to \mathbb{P}^2$. So $X\sim E+F+H$ in $\mathbb{F}_1\times\mathbb{P}^2$, where $E$ and $F$ are the pullbacks of a $(-1)$-section and a fiber on $\mathbb{F}_1$, and $H$ is the pullback of a line from $\mathbb{P}^2$. Alternatively, $X$ can be described as $\mathbb{P}_{\mathbb{F}_1}(f^*T_{\mathbb{P}^2})$, where $f\colon \mathbb{F}_1\to\mathbb{P}^2$ is a blow up of a point. Note that $X$ is a Fano threefold $3.24$ as in \cite[12.3]{IP99}.

Note that $X$ can be blown down to $\mathbb{P}^1\times\mathbb{P}^2$ as follows: $f_1\colon X\subset \mathbb{F}_1\times\mathbb{P}^2\to Z=\mathbb{P}^1\times\mathbb{P}^2$. This map is birational and contracts $D_1\simeq \mathbb{F}_1$. Also, $X$ can be blown down to $\mathbb{P}^2\times\mathbb{P}^2$ as follows: $f_2\colon X\subset \mathbb{F}_1\times\mathbb{P}^2\to Y\subset \mathbb{P}^2\times\mathbb{P}^2$. This map is birational and contracts $D_2\simeq\mathbb{P}^1\times\mathbb{P}^1$. We have 
\[
K_X 
\sim -E-2F-2H|_X.
\]

We have $D_1\sim H-E$, $D_2\sim E$. Then 
$
L=-K_X -aD_1-bD_2 \sim (1-b+a)E+2F+(2-a)H.
$ 
There are $3$ extremal contractions: $f_1$, $f_2$ which are birational, and $f_3\colon X \to \mathbb{F}_1$, which is of fiber type. The corresponding extremal curves are 
\[
l_1 = (H-E)H , \quad \quad \quad l_2= E\cdot H, \quad \quad \quad l_3= (F+E)^2.
\]

Compute
$L\cdot l_1 
= 1-b+a>0$, $L\cdot l_2 = 1+b-a>0$, 
$L\cdot l_3 = 2-a$. Hence $(X, \Delta)$ is log Fano for any $0\leq a, b<1$. Compute the intersection theory: $H^3=0$, $E^2F=EF^2=E^3=F^3=0$, $F H^2 
=1$, 
$EH^2 
= 0$, 
$E^2H 
= -1$, 
$EFH 
=1$.
We have
\[
L - tD_2 \sim (1-b+a-t)E+2F+(2-a)H = (3-b-t)E+2F+(2-a)(H-E).
\]
 
Consequently, 
$
(L-tD_2)\cdot l_1 
= 1-b+a-t.
$
$
(L-tD_2)\cdot l_2 
= 1+b-a+t.
$
$
(L-tD_2)\cdot l_3 = 2 - a.
$
The nef threshold for $D_2\sim E$ is equal to $1-b+a$, and the pseudo-effective threshold is equal to $3-b$. Write down Zariski decomposition for 
\begin{multline*}
L - tD_2 \sim -K_X -aD_1-bD_2 = (1-b+a - t)E+2F+(2 - a)H = (3-b- t)E+2F+(2 - a)(H-E) \\
2F+(3-b- t)H + (1-b+a+t)(H-E)
. \\
= 2F+(2 - a - t')H + t'(H-E)
\end{multline*}

Compute
\begin{multline*}
S_1 = \int_0^{1-b+a}{\bigg(\big(1-b+a - t\big)E+2F+\big(2 - a\big)H\bigg)^3 dt} \\ 
= \int_0^{1-b+a}{\bigg(12\big(1-b+a - t\big)\big(2 - a\big) + 6\big(2 - a \big)^2 - 3 \big(1-b+a - t\big)^2\big(2 - a\big)\bigg) dt} \\
= \big(2-a\big) \bigg(17-a^3-3 a^2 \big(1-b\big)-21 b+3 b^2+b^3+3 a \big(5-b^2\big)\bigg).
\end{multline*}
\begin{multline*}
S_2 = \int_{1-b+a}^{3-b}{\bigg(2F+\big(3-b- t\big)H\bigg)^3 dt} 
= \int_{1-b+a}^{3-b}{6\big(3-b- t\big)^2 dt} = 2 \big(2-a\big)^3. \\
\end{multline*}

Finally, compute
\begin{multline*}
\beta'_{(X, D)}(D_2) = \big(1-b\big)\bigg(12\big(1-b+a\big)\big(2-a\big)+6\big(2-a\big)^2 -3\big(1-b+a\big)^2\big(2-a\big)\bigg) \\
-\big(-2+a\big) \bigg(-17+a^3-3 a^2 \big(-1+b\big)+21 b-3 b^2-b^3+3 a \big(-5+b^2\big)\bigg) \\
-2 \big(2-a\big)^3=\big(2-a\big) \bigg(-a^2\big(2 -a\big)-2 \big(2-b\big) \big(1+b\big)^2-a \big(7-6 b+3 b^2\big)\bigg)\leq -2\big(2-a\big) \big(2-b\big) \big(1+b\big)^2<0.
\end{multline*}

Hence, we are done in this case.

\section{Blow up of a $\mathbb{P}^2$-bundle over $\mathbb{P}^1$}
\label{sec-blow-up-of-p2-bundle}
In this section, we prove the following
\begin{proposition}
\label{prop-unstable-blow-up-of-p2-bundles}
Consider a log Fano pair of Maeda type $(X, D)$ as in the case \ref{X-is-Fano-blow-up-of-P2-bundle} of Corollary \ref{X-classification}. Then it is divisorially unstable for any values of the parameters.
\end{proposition}
\begin{proof}
Here $X$ is the blow up of a point on a $\mathbb{P}^2$-bundle $Y= \mathbb{P}_{\mathbb{P}^1}\big(\oo\oplus\oo\oplus\oo(n)\big)$ with $n\geq 1$, and $D=aD_1+bD_2$ for $0\leq a,b<1$.
Consider the blow down morphism $f\colon X \to Y$, and for the boundary on $Y$ we have $\Delta_Y = D'_1 + D'_2$ where $D'_i=f(D_i)$. 
We have that $\Delta|_E$ is a line on $E\simeq \mathbb{P}^2$, $D_1\simeq \mathbb{F}_1$, and $D_2\simeq D'_2 \simeq \mathbb{P}^1\times\mathbb{P}^1$. Further, we have
\[
K_Y = -3H - \big(2-n\big)F, \ \ \ \ \ \ D'_1 \sim F, \ \ \ \ \ \ D'_2 \sim H - n F, 
\]

Note that
$
-K_Y - aD'_1 - bD'_2 \sim \big(3-b\big)H + \big(2-n\big(1-b\big)-a\big)F
$
is ample if and only of $n(1-b)+a<2$. We have
$f^*D'_1 = D_1 + E$ and $f^* D'_2 = D_2$.
One checks that  $\mathrm{Pic}(X)=\langle D_1, D_2, E\rangle$, 
$\overline{\mathrm{Eff}}(X) = \mathbb{R}_{\geq0}\big[E\big] + \mathbb{R}_{\geq0}\big[D_1\big]+ \mathbb{R}_{\geq0}\big[D_2\big]$, and
\begin{equation*}
\begin{split}
\mathrm{Nef}(X) &= \mathbb{R}_{\geq0}\big[E+D_1\big] + \mathbb{R}_{\geq0}\big[n E + n D_1 + D_2\big]+ \mathbb{R}_{\geq0} \big[(n-1)E + nD_1 + D_2\big]. \\
\end{split}
\end{equation*}

We compute $\beta'_{(X, D)}(D_2)$. Then $A_{(X, D)}(D_1)=1-b$. We have
$
L - xD_2 \sim -K_X -D - xD_2  
=\big(2+2n-a\big)D_1 + \big(3-b-x\big)D_2 + 2nE.
$ Hence, the pseudo-effective threshold is equal to $3-b$. 
Write
\begin{multline*}
L - xD_2 \sim \big(2 - a - n \big(1-b-x\big)\big) \big(E+D_1\big) \\+ \big(1 + a - b - x\big) \big(nE+nD_1+D_2\big) 
+ \big(2-a\big) \big(\big(n-1\big)E+nD_1+D_2\big) \\
= \big(2  - \big(n-1\big) \big(1-b-x\big)\big) \big(E+D_1\big) + \big(3 - b - x \big) \big(\big(n-1\big)E+nD_1+D_2\big) \\
- \big(1+a-b-x\big) D_1.
\end{multline*}

Thus, the nef threshold for $L-xD_2$ is equal to $1+a-b$. For $x=1+a-b$, $L-xD_2$ defines a contraction $g\colon X \to Z= \mathbb{P}_{\mathbb{P}^1}(\oo\oplus\oo\oplus\oo(n-1))$. For $1+a-b-x\leq x\leq 3-b$ we have the Zariski decomposition $L-xD_2 = P(x)+N(x)$ where the negative part is $N(x)=(x-1-a+b) D_1$, and for the positive part $P(x)$ we have 
\begin{multline*}
P(X) = \big(2 - \big(n-1\big) \big(1-b-x\big)\big) \big(E+D_1\big) + \big(3 - b - x \big) \big(\big(n-1\big)E+nD_1+D_2\big) \\
= g^* \big(\big(2 - \big(n-1\big) \big(1-b-x\big)\big) E + \big(3 - b - x \big) \big(\big(n-1\big)E+D_2\big)\big)
\end{multline*}
where we identify $E$ and $D_2$ with their images on $Z$. Compute 
\begin{multline*}
S_1 
= \int_0^{1+a-b}{\bigg(f^*\Big(\big(3-b-x\big)H+\big(2-n\big(1-b-x\big)-a\big)F\Big)^3 - \big(2-a\big)^3\bigg)dx} \\
= \big(1+a-b\big) \bigg(60+b^2 \big(4-7 n\big)+a^2 \Big(6+b \big(-2+n\big)-5 n\Big) \\
+ 11 n+a^3 n+b^3 n+a \Big(-42+b \big(20-6 n\big)+b^2 \big(-2+n\big)+5 n\Big)+b \big(-32+11 n\big)\bigg)/2.
\end{multline*}

We compute the second part of the volume function on $Z$: 
\begin{multline*}
S_2 
= \int_{1+a-b}^{3-b}{\bigg(\big(2  - \big(n-1\big) \big(1-b-x\big)\big)E + \big(3 - b - x \big)\big(\big(n-1\big)E+D_2\big)\bigg)^3 dx} \\
= \big(2-a\big)^3 \Big(a \big(-1+c\big)+2 \big(1+c\big)\Big)/2.
\end{multline*}

Finally, we have
\begin{multline*}
\beta'_{(X, D)}(D_2) = \big(1-b\big)L^3 - S_1 - S_2 \\
=\big(1-b\big) \Big(\big(3-b\big)^3 n+ 3 \big(3-b\big)^2\big(2-n\big(1-b\big)-a\Big) - \big(2-a\big)^3\Big) - S_1 - S_2 \\
= \bigg(16+4 a^3-a^4-4 a \big(4-b\big) \big(1-b\big)^2-36 b \big(2-n\big) \\
-27 n-3 b^4 n+4 b^3 \big(5 n-2\big)-6 b^2 \big(7 n-8\big)\bigg)/2.
\end{multline*}
We show that this expression is negative for $0\leq a, b<1$. 
Indeed, for $n\geq 2$ we have that is less or equal than 
\[
\big(16 + 4 -27n + 4 \big(5n-2\big)\big)/2=\big(12-7n\big)/2<0, 
\]
and for $n=1$ it is less or equal than 
\begin{multline*}
\bigg(16+4 a^3-a^4-4 a \big(4-b\big) \big(1-b\big)^2-36 b -27-3 b^4+12 b^3 + 6 b^2 \bigg)/2 \\
= \bigg(16+3 a^3+a^3\big(1-a\big)-4 a \big(4-b\big) \big(1-b\big)^2 -27+3 b^3\big(1-b\big)-9 b\big(1-b^2\big)- 6 b\big(1-b\big)-21b \bigg)/2 \\
\leq \bigg(16+3 a^3+a^2/4-4 a \big(4-b\big) \big(1-b\big)^2 -27+3 b^2/4-9 b\big(1-b^2\big) - 6 b\big(1-b\big)-21b \bigg)/2 \\
\leq \bigg(16+3+1/4 -27+3/4 \bigg)/2<0.
\end{multline*}

Hence, the pair $(X, D)$ is K-unstable, and we are done in this case.\end{proof}

\section{Fano threefolds}
\label{sec-fano}
In this section, we assume that $X$ is a Fano threefold with $\rho(X)=1$ such that it is a log Fano with respect to a reducible boundary, that is $(X, \Delta)$ is such that $\Delta$ has $\geq 2$ components. Hence, for the Fano index we have $i(X)\geq 3$. Thus, $X$ is either a projective space $\mathbb{P}^3$, or a quadric $Q_3$.
\subsection{Projective space}
\label{subsec-p3}
In this subsection, we prove the following

\begin{proposition}
\label{prop-p3}
Let $(\mathbb{P}^3, D)$ be a log Fano pair of Maeda type such that $D$ is reducible. Then it is K-semistable if and only if $D=aD_1+bD_2$ where $D_1$ is a quadric, $D_2$ is a hyperplane, and $2a\geq 3b$, $6a-b\geq 4$. Moreover, for $2a> 3b$ and $6a-b> 4$ this pair is K-polystable.
\end{proposition}
\begin{proof}
The case of $(\mathbb{P}^3, \Delta)$ such that the components of $\Delta$ are hyperplanes was treated in \cite[1.6]{Fu20b}, and such pairs are K-unstable for any values of the coefficients such that $D\neq 0$. If $\Delta$ is reducible and not all of its components are hyperplanes, then there is only the following case to consider: $\big(\mathbb{P}^3, \Delta)$ where $\Delta = D_1 + D_2$ with $D_1$ a smooth quadric and $D_2$ a hyperplane. 

Put $D = aD_1 + bD_2$, $0\leq a, b <1$. Then $L=-K_X-D \sim \big(4-2a-b\big)H$ is ample. We have $L^3 = \big(4-2a-b\big)^3$. We estimate $\beta'_{(X, D)}(D_2)$. 
We have $A_{(X, D)}(D_2) = 1-b$. Then
\begin{equation*}
\begin{split}
S'_{(X, D)}(D_2)=&\int_0^{4-2a-b}{\bigg(L-xD_2\bigg)^3 dx} 
= \big(4-2a-b\big)^4/4.
\end{split}
\end{equation*}

We have
$
\beta'_{(X, D)}(D_2) = \big(1-b\big)\big(4-2a-b\big)^3 - \big(4-2a-b\big)^4/4 
= \big(4-2a-b\big)^3 \big( a/2 - 3b/4 \big) 
$
which is negative for $2a<3b$. Now, compute $\beta'_{(X, D)}(D_1)$. 
We have $A_{(X, D)}(D_1) = 1-a$. Then
$
S'_{(X, D)}(D_1) 
= \big(4-2a-b\big)^4/8.
$
Hence
$
\beta'_{(X, D)}(D_1) = \big(1-a\big)\big(4-2a-b\big)^3 - \big(4-2a-b\big)^4/8 
= \big(4-2a-b\big)^3 \big( 1/2 - 3a/4 + b/8 \big)
$
which is negative for $6a-b>4$, or $a>2/3+b/6$. 

Now we check divisorial stability for the divisors not equal to the components of $\Delta$. By Lemma 9.1 and Remark 9.6 in \cite{Fu16a} (they are formulated for the anticanonical polarisation, but work in our setting as well) it is enough to compute $\beta'_{(X, D)}(D')$ for $D'\sim H$. However, it is clear that $\beta'_{(X, D)}(D_2)\leq \beta'_{(X, D)}(D')$, hence the condition $\beta'_{(X, D)}(D_i)\geq 0$ for $i=1, 2$ is sufficient for divisorial stability.
Hence, for $6a-b\leq4$, $2a\leq3b$ the pair $(X, D)$ is divisorially stable. Now, we check K-semistability using Theorem \ref{thm-G-invariant}. Note that for the the automorphism group $G$ of the pair $(X, \Delta)$ we have $G\simeq \mathrm{PGL}(2, \mathbb{C})\times\mathbb{Z}/2$. In particular, $G$ is reductive. Thus, proper $G$-invariant subsets of $X$ are as follows: $D_1, D_2, Z=D_1\cap D_2$.

Hence, to check K-polystability, it is enough to compute $\beta$-invariant only for $G$-invariant divisors over $X$ whose center on $X$ is equal to $Z$. To this aim, we use Proposition \ref{prop-formulas-for-beta}. Let $Y=D_1$. According to Proposition \ref{prop-formulas-for-beta}, first estimate $\frac{A_{(X, D)}(Y)}{S_{(X, D)}(Y)}>1$ for $2a> 3b$ and $6a-b> 4$.  

Note that $Z$ is a curve of bidegree $(1, 1)$ on $Y$. Now we compute $S(W_{\bullet, \bullet}; Z)$ using the formula \eqref{eq-min2}. Since $N(u)=0$, we have 
\begin{multline*}
S(W^Y_{\bullet, \bullet}; Z) 
= \frac{3}{L^3} \int_0^{\infty}\int_0^{\infty}{\mathrm{vol}\Big((L-uY)|_Y - v Z\Big)dv du} \\
= \frac{3}{L^3} \int_0^{\infty}\int_0^{4-2a-b-2u}{\mathrm{vol}\Big(\big(4-2a-b-2u-v\big)\big(1, 1\big)\Big)dv du} \\
= \frac{3}{L^3} \int_0^{\infty}\int_0^{4-2a-b-2u}{ 2\big(4-2a-b-2u-v\big)^2dv du} 
= \frac{\big(2-a-b/2\big)}{2}.
\end{multline*}

By Proposition \ref{prop-formulas-for-beta} we have 
$
\delta_Z(X, D)\geq \frac{1-a}{\big(1-a/2-b/4\big)} = \frac{A_{(X, D)}(Y)}{S_{(X, D)}(Y)} >1
$
for $2a> 3b$ and $6a-b> 4$. We conclude that the pair is K-semistable (resp., K-polystable) for $2a\geq 3b$ and $6a-b\geq 4$ (resp., for $2a> 3b$ and $6a-b> 4$). Hence, we are done in this case.
\end{proof}

\subsection{Quadric hypersurface} 
\label{subsec-quadric}
Consider a log smooth log Fano pair $(X, \Delta)$ where $X$ is a smooth quadric $Q\subset \mathbb{P}^4$ and $\Delta = D_1 + D_2$ with $D_i$ being its hyperplane sections. Put $D=aD_1 + bD_2$. In this subsection, we prove the following

\begin{proposition}
\label{prop-quadric}
The pair $(X, D)$ is K-semistable if and only if $1+a-3b\geq0$ and $1+b-3a\geq0$. Moreover, for $1+a-3b>0$ and $1+b-3a>0$ this pair is K-polystable.
\end{proposition}
\begin{proof}
We have $L = -K_X -D \sim (3-a-b)H$. 
Compute $\beta'_{(X, D)}(D_1)$. We have $A_{(X, D)}(D_1) = 1-a$. Then
\begin{equation*}
\begin{split}
S'_{(X, D)}(D_1) = \int_0^{3-a-b}{\bigg(L-xD_1\bigg)^3dx} 
= \big(3-a-b\big)^4/2.
\end{split}
\end{equation*}

We have
$
\beta'_{(X, D)}(D_1) = 2\big(1-a\big)\big(3-a-b\big)^3 - \big(3-a-b\big)^4/2 
= \big(3-a-b\big)^3\big(1 + b - 3a \big)/2.
$
Similarly, $\beta'_{(X, D)}(D_2)=\big(3-a-b\big)^3\big(1 + a - 3b \big)/2$. Now we check divisorial stability for divisors not equal to the components of $\Delta$. By Lemma 9.1 and Remark 9.6 in \cite{Fu16a} (they are formulated for the anticanonical polarisation, but work in our setting as well) it is enough to compute $\beta'_{(X, D)}(D')$ for $D'\sim H$. However, it is clear that $\beta'_{(X, D)}(D_1)\leq \beta'_{(X, D)}(D')$, hence the condition $\beta'_{(X, D)}(D_i)\geq 0$ for $i=1, 2$ is sufficient for divisorial stability.
Thus $(X, D)$ is divisorially stable when $1 + a - 3b\geq0$ and $1 + b - 3a\geq0$.

As in the previous case, we see that there exists a subgroup $G\simeq \mathrm{PGL}(2, \mathbb{C})$ in the automorphism group of the pair $(X, \Delta)$ which faithfully acts on $Z=D_1\cap D_2\simeq \mathbb{P}^1$. To check K-polystability, by Theorem \ref{thm-G-invariant} it is enough to compute $\beta$-invariant only for $G$-invariant centers of divisors over $X$. Hence, by the above we may consider only the intersection $Z=D_1\cap D_2$. Let $Y=D_2$. To apply Proposition \ref{prop-formulas-for-beta}, we estimate
$
\frac{A_{(X, D)}(Y)}{S_{(X, D)}(Y)} >1 
$
for $1 + a - 3b>0$ and $1 + b - 3a>0$.

Note that $Y$ is a quadric, and $Z$ is a curve of bidegree $(1, 1)$. We compute $S(W_{\bullet, \bullet}; Z)$ using the formula \eqref{eq-min2}. Since $N(u)=0$, we have 
\begin{multline*}
S(W^Y_{\bullet, \bullet}; Z) 
= \frac{3}{L^3} \int_0^{\infty}\int_0^{\infty}{\mathrm{vol}\Big((L-uY)|_Y - v Z\Big)dv du} \\
= \frac{3}{L^3} \int_0^{\infty}\int_0^{3-a-b-u}{\mathrm{vol}\Big(\big(3-a-b-u-v\big)\big(1, 1\big)\Big)dv du} \\
= \frac{3}{L^3} \int_0^{\infty}\int_0^{3-a-b-u}{ 2\big(3-a-b-u-v\big)^2dv du} 
= \frac{\big(3-a-b\big)}{4}.
\end{multline*} 

By Proposition \ref{prop-formulas-for-beta} we obtain
$
\delta_Z(X, D)\geq \frac{4(1-a)}{3-a-b}=\frac{A_{(X, D)}(Y)}{S_{(X, D)}(Y)}>1
$
for $1+a-3b>0$ and $1+b-3a>0$. We conclude that the pair is K-semistable (resp., K-polystable) for $1+a-3b\geq0$ and $1+b-3a\geq0$ (resp., $1+a-3b>0$ and $1+b-3a>0$). Hence, we are done in this case.
\end{proof}

\section{Proof of main theorems}
\label{sec-proof-main}

\begin{proof}[Proof of Theorem \ref{main-thm2}]
We use the classification of the pairs $(X, \Delta)$ with $\dim X=3$, due to Maeda \cite{M83}. We assume that $\Delta$ is reducible. The claim on the MMP with scaling follows from Proposition \ref{prop-MMP}. In Corollary \ref{X-classification}, we obtain the list of possibilities for $X$ and $\Delta$ according to the type of the first  contraction in the MMP. In sections \ref{sec-blow-up-of-quadric}--\ref{sec-fano} we list all the possibilities for these pairs. We summarise all the cases in Table \hyperlink{table-1}{$1$}. 
\end{proof}

\begin{proof}[Proof of Theorem \ref{main-thm}]
Let $(X, D)$ be a log Fano pair of Maeda type, such that $\dim X = 3$ and $D$ is reducible. Let $(X, \Delta)$ be the corresponding log smooth log Fano pair. First we show that the pair $(X, D)$ is divisorially unstable in all cases, except cases listed in Theorem \ref{main-thm}. To this aim, we use the classification of such pairs as in Theorem \ref{main-thm2} according to the type of contraction in the MMP as in Proposition \ref{prop-MMP}.

In section \ref{sec-blow-up-of-quadric} we consider the blow up of a quadric. In Proposition \ref{prop-unstable-blow-up-of-quadric} we show that $(X, D)$ is divisorially unstable for any values for the parameters. In section \ref{section-p2-bundles} we consider $\mathbb{P}^2$-bundles over $\mathbb{P}^1$. In Proposition \ref{prop-unstable-p2bundles} we show that $(X, D)$ is divisorially unstable for any values for the parameters, unless $(X, D)$ is as in the case \hyperlink{compute-p2-bundle-cases-v}{D5}. In section \ref{section-quadric-bundles} we consider quadric bundles over $\mathbb{P}^1$, and show that $(X, D)$ is divisorially unstable for any values for the parameters, unless $(X, D)$ is as in the case \hyperlink{compute-quadric-bundle-case-1}{Q1} with $m=1$. 
In section \ref{section-p1-bundles} we consider $\mathbb{P}^1$-bundles over a surface. In Proposition \ref{prop-unstable-p1-bundles} we show that $(X, D)$ is divisorially unstable for any values for the parameters, unless $(X, D)$ is as in the case \hyperlink{compute-p1-bundle-cases-8}{C9}. 
In section \ref{sec-blow-up-of-p2-bundle} we consider the blow up of a point on certain $\mathbb{P}^1$-bundles over $\mathbb{P}^1$. In Proposition \ref{prop-unstable-blow-up-of-p2-bundles} we show that $(X, D)$ is divisorially unstable for any values for the parameters. Finally, in section \ref{sec-fano} we consider Fano threefolds. In Propositions \ref{prop-p3} and \ref{prop-quadric} we determine the values of the parameters for which the corresponding pairs are K-semistable. For each of divisorially stable pairs, using Propositions \ref{prop-formulas-for-beta} and \ref{thm-G-invariant}, we show that divisorial stability implies K-polystabiliy. This concludes the proof of the theorem.
\end{proof}

\begin{appendix}
\section{The table}
\label{sect-table}
\end{appendix}
We list $3$-dimensional log smooth log Fano pairs $(X, \Delta)$ where $\Delta=\sum D_i$ with the property that $\Delta$ is reducible. In the last column, we indicate whether a log Fano pair $(X, D)$ of Maeda type can be K-semistable, where $\supp (D) \subset \supp (\Delta)$. As a consequence, we note that if $\Delta$ is reducible, then $X$ is rational, it has no moduli, and admits a faithful action of a $2$-dimensional torus.

We use the following notation. In cases \hyperref[sec-blow-up-of-quadric]{E1} and \hyperref[subsec-quadric]{F4}, $Q$ is a quadric hypersurface in $\mathbb{P}^4$. In cases \hyperref[sec-blow-up-of-quadric]{E1} and \hyperref[sec-blow-up-of-p2-bundle]{E2}, $E\simeq \mathbb{P}^2$ is the exceptional divisor of the blow up in a point. In case \hyperref[sec-blow-up-of-quadric]{E1}, $H\;\widetilde{}\ $ is the strict transform of a hyperplane in $\mathbb{P}^4$ that passes through the blown up point $p$. In case \hyperref[sec-blow-up-of-p2-bundle]{E2}, the blown-up point $p$ does not belong to the unique divisor, linearly equivalent to $H-a_2F$ on $Y=\mathbb{P}_{\mathbb{P}^1}(\oo\oplus\oo\oplus\oo(a_2))$. Then, $(H-a_2F)\;\widetilde{}\ $ is its strict transform on $X$, and $F\;\widetilde{}\ $ is the strict transform on $X$ of the fiber $F$ of the morphism $Y\to \mathbb{P}^1$ with the condition that $F$ contains the point $p$. In cases \hyperlink{compute-p1-bundle-cases-remark}{C1}--\hyperlink{compute-p2-bundle-cases-viii}{D8}, $H$ is the tautological divisor. In cases \hyperlink{compute-p2-bundle-cases-remark}{D1}--\hyperlink{compute-p2-bundle-cases-viii}{D8}, $F$ is a fiber of the morphism $X\to \mathbb{P}^1$. In cases \hyperlink{compute-p1-bundle-cases-remark}{C1}--\hyperlink{compute-p1-bundle-cases-iii}{C3}, $F$ is a preimage of a line on $\mathbb{P}^2$ via the morphism $X\to \mathbb{P}^2$. In cases \hyperlink{compute-p1-bundle-cases-iv}{C4}, \hyperlink{compute-p1-bundle-cases-vii}{C5} and \hyperlink{compute-p1-bundle-cases-vii}{C8}, $F_f$ (resp., $F_s$) is a preimage of a fiber $f$ (resp., of a section $s$ such that $s^2=-n$) on $\mathbb{F}_n$ via the morphism $X\to \mathbb{F}_n$. In case \hyperlink{compute-p1-bundle-cases-viii}{C7}, $F_1$ and $F_2$ are preimages of two different rulings via the morphism $X\to \mathbb{P}^1\times \mathbb{P}^1$. In case \hyperlink{compute-p1-bundle-cases-8}{C9}, $H_1$ and $H_2$ are divisors of bidegree $(1, 0)$ and $(0, 1)$. In case \hyperlink{compute-p1-bundle-cases-9}{C10}, $H$ is the pullback of a line under the natural morphism $X\to \mathbb{P}^2$, $F$ and $E$ are preimage of a ruling and a $(-1)$-curve under the natural morphism $X\to \mathbb{F}_1$. Finally, in all cases $\epsilon$ is the nef value, see Remark \ref{rem-nef-value}.

\hypertarget{table-1}
\bigskip
\begin{adjustbox}{width=1.12\textwidth,center=\textwidth}
\begin{tabular}{ | l | l | l | l | l | l | l | l | l | }
\hline
& X & & $[D_1]$ & $[D_2]$ & $[D_3]$ & $\epsilon$ & No in \cite{M83} & K-semistable \\
\hline
\hyperref[sec-blow-up-of-quadric]{E1} & $\mathrm{Bl}_p\,Q$ & $\vphantom{a^{b^b}}$ & $E$ & $H\:\,\widetilde{}$ & & $2$ & 7.1 & No \\ 
\hline
\hyperref[sec-blow-up-of-p2-bundle]{E2} & $\mathrm{Bl}_p\, \mathbb{P}_{\mathbb{P}^1}(\oo\oplus\oo\oplus\oo(n))$ & $n\geq 1$ & $(H-nF)\:\widetilde{}$ & $F\:\,\widetilde{}$ & $\vphantom{A^{b^b}}$ & $2$ & 7.4 & No \\ 
\hline
\hyperlink{compute-p1-bundle-cases-remark}{C1} & $\mathbb{P}_{\mathbb{P}^2}(\oo\oplus\oo(k))$ & $k\geq0$ & $H-kF$ & $2F$ & & $2$ & 8.2\,(1)& No \\ 
\hline
\hyperlink{compute-p1-bundle-cases-ii}{C2} & $\mathbb{P}_{\mathbb{P}^2}(\oo\oplus\oo(k))$ & $k\geq0$  & $H-kF$ & $F$ & $F$ & $2$ & 8.2\,(2)& No \\ 
\hline
\hyperlink{compute-p1-bundle-cases-iii}{C3} & $\mathbb{P}_{\mathbb{P}^2}(\oo\oplus\oo(k))$  & $k\geq -1$ & $H-kF$ & $F$ & & $2$ & 8.2\,(3)& No \\ 
\hline
\hyperlink{compute-p1-bundle-cases-iv}{C4} & $\mathbb{P}_{\mathbb{F}_{n}}(\oo\oplus\oo(ks+(kn+m)f)$ & $k, n\geq 0$, $m\geq -1$ & $H-kF_s-(k n+m)F_f$ & $F_s$ & & $2$ & 8.3\,(1)--8.3\,(3)& No \\ 
\hline
\hyperlink{compute-p1-bundle-cases-vii}{C5} & $\mathbb{P}_{\mathbb{F}_{n}}(\oo\oplus\oo(ks+(kn+m)f))$ & $k, m, n\geq 0$ & $H-kF_s-(k n+m)F_f$ & $F_s$ & $F_f$ & $2$ & 8.4& No \\ 
\hline
\makecell[l]{\hyperlink{compute-p1-bundle-cases-new}{C6}\\ \vphantom{}} & \makecell[l]{$\mathbb{P}_{\mathbb{F}_n}(\ee)$\\ $[\ee]\in \mathrm{Ext}^1 (\oo, \oo(ks+(kn-2)f))-\{0\}$} & \makecell[l]{$k, n\geq 1$\\ $(k, n)\neq(1,1)$} & \makecell[l]{$H$\\ \vphantom{}} & \makecell[l]{$F_s$\\ \vphantom{}} & & \makecell[l]{$2$\\ \vphantom{}} & \makecell[l]{--\\ \vphantom{}} & \makecell[l]{No\\ \vphantom{}} \\ 
\hline
\hyperlink{compute-p1-bundle-cases-viii}{C7} & $\mathbb{P}_{\mathbb{P}^1\times\mathbb{P}^1}(\oo\oplus\oo(k, n))$ & $k, n\geq 0$ & $H-kF_1-nF_2$ & $F_1+F_2$ & & $2$ & 8.5& No \\ 
\hline
\hyperlink{compute-p1-bundle-cases-vii}{C8} & {$\mathbb{P}_{\mathbb{F}_{1}}(\oo\oplus\oo(ks+(k+m)f))$} & $k, m\geq 0$ & $H-kF_s-(k+m)F_f$ & $F_h$ & & $2$ & 8.6 & No \\ 
\hline
\hyperlink{compute-p1-bundle-cases-8}{C9} & {$\mathbb{P}_{\mathbb{P}^2}(T_{\mathbb{P}^2})$} &  & $H_1$ & $H_2$ & & 2 & -- & For some $a, b$ \\ 
\hline
\makecell[l]{\hyperlink{compute-p1-bundle-cases-9}{C10}\\ \vphantom{}} & \makecell[l]{$X\subset\mathbb{F}_1\times\mathbb{P}^2$\\$X\sim F+E+H$} & & \makecell[l]{$H-E$\\ \vphantom{}}  & \makecell[l]{$E$\\ \vphantom{}}  & & \makecell[l]{$2$\\ \vphantom{}}  & \makecell[l]{--\\ \vphantom{}} & \makecell[l]{No\\ \vphantom{}}  \\ 
\hline
\hyperlink{compute-p2-bundle-cases-remark}{D1} & $\mathbb{P}_{\mathbb{P}^1}(\oo\oplus\oo(k)\oplus\oo(n))$ & $k, n\geq 0$ & $H-kF$ & $H-nF$ & & $3$ & 9.1.3\,(i)& No \\ 
\hline
\hyperlink{compute-p2-bundle-cases-ii}{D2} & $\mathbb{P}_{\mathbb{P}^1}(\oo\oplus\oo(1)\oplus\oo(n))$ & $n\geq 1$  & $H$ & $H-nF$ & & $3$ & 9.1.3\,(ii)& No\\ 
\hline
\hyperlink{compute-p2-bundle-cases-remark}{D3} & $\mathbb{P}_{\mathbb{P}^1}(\oo\oplus\oo\oplus\oo(1))$ &  & $H$ & $H$ & & $3$ & 9.1.3\,(iii)& No \\ 
\hline
\hyperlink{compute-p2-bundle-cases-iv}{D4} & $\mathbb{P}_{\mathbb{P}^1}(\oo\oplus\oo\oplus\oo(n))$ & $n\geq 1$ & $H+F$ & $H-nF$ & & $3$ & 9.1.3\,(iv)& No\\ 
\hline
\hyperlink{compute-p2-bundle-cases-v}{D5} & $\mathbb{P}_{\mathbb{P}^1}(\oo\oplus\oo\oplus\oo)$ & & $H+F$ & $H$ & & $3$ & 9.1.3\,(v)& For some $a, b$ \\ 
\hline
\hyperlink{compute-p2-bundle-cases-remark}{D6} & $\mathbb{P}_{\mathbb{P}^1}(\oo\oplus\oo\oplus\oo(n))$ & $n\geq 1$ & $H-nF$ & $F$ & & $3/2$ & 9.1.4\,(i)& No\\ 
\hline
\hyperlink{compute-p2-bundle-cases-vii}{D7} & $\mathbb{P}_{\mathbb{P}^1}(\oo\oplus\oo\oplus\oo)$ & & $2H$ & $F$ & & $3$ & 9.1.4\,(ii)& No\\ 
\hline
\hyperlink{compute-p2-bundle-cases-viii}{D8} & $\mathbb{P}_{\mathbb{P}^1}(\oo\oplus\oo(k)\oplus\oo(n))$ & $k, n\geq 0$ & $H-kF$ & $H-nF$ & $F$ & $3$ & 9.1.4\,(iii)& No \\ 
\hline
\makecell{\hyperlink{compute-quadric-bundle-case-1}{Q1 }\\ \vphantom{}} & \makecell[l]{$X\subset \mathbb{P}_{\mathbb{P}^1}(\oo\oplus\oo\oplus\oo\oplus\oo(m))$\\$X\sim2H$} & \makecell[l]{$m\geq 1$\\ \vphantom{}} & \makecell[l]{$H-mF$\\ \vphantom{}} & \makecell[l]{$F$\\ \vphantom{}} & & \makecell[l]{$2$\\ \vphantom{}} & \makecell[l]{9.2\,(2)\\ for $m=1$} & \makecell[l]{For $m=1$\\ and some $a, b$} \\ 
\hline
\hyperref[subsec-p3]{F1} & $\mathbb{P}^3$ & & $H$ & $H$ & $H$ & $4$ & 6.1\,(iii) & No \\ 
\hline
\hyperref[subsec-p3]{F2} & $\mathbb{P}^3$ & & $H$ & $H$ & & $2$ & 6.1\,(v) & No \\ 
\hline
\hyperref[subsec-p3]{F3} & $\mathbb{P}^3$ & & $2H$ & $H$ & & $4/3$ & 6.1\,(ii) & For some $a, b$ \\ 
\hline
\hyperref[subsec-quadric]{F4} & $Q$ & & $H$ & $H$ & & $3$ & 6.2 & For some $a, b$ \\ 
\hline
\end{tabular}
\end{adjustbox}
\smallskip
\

\begin{center}
\textit{Table $1$. Log smooth log Fano threefolds $(X, \Delta)$ with integral reducible boundary $\Delta$}. 
\end{center}
\bigskip


\def\cprime{$'$} \def\mathbb#1{\mathbf#1}


\begin{thebibliography}{}

\bibitem[AZ20]{AZ20}
H. Abban, Z. Zhuang,
\newblock {\em Stability of Fano varieties via admissible flags.} 
\newblock preprint, arXiv:2003.13788, 2020.

\bibitem[ACC$^+$]{ACCFKMSSV21}
C. Araujo, A.-M. Castravet, I. Cheltsov, K. Fujita, A.-S. Kaloghiros, J. Martinez-Garcia, C. Shramov, H. S\"u\ss, N. Viswanathan.
\newblock {\em The Calabi problem for Fano threefolds.} 
\newblock MPIM preprint, 2021.

\bibitem[BS97]{BS97}
M. Beltrametti, A. Sommese.
\newblock {\em The adjunction theory of complex projective varieties}.
\newblock De Gruyter Expositions in Mathematics, vol. 16, Walter de Gruyter and Co., Berlin, 1995.

\bibitem[BCHM10]{BCHM10}
C. Birkar, P. Cascini, C. Hacon, J. McKernan.
\newblock {\em Existence of minimal models for varieties of log general type}.
\newblock Journal of Amer. Math. Soc., 23, 2, 405--467 (2010).

\bibitem[Ca12]{Ca12}
C. Casagrande.
\newblock {\em On the Picard number of divisors in Fano manifolds.} 
\newblock Annales scientifiques de l'Ecole Normale Superieure, Serie 4, Tome 45 no. 3, 363--403 (2012).

\bibitem[CMR20]{CMR20}
P. Cascini, J. Martinez-Garcia, Y. Rubinstein.
\newblock {\em On the body of ample angles of asymptotically log Fano varieties.} 
\newblock arXiv:2007.15595, 2020.

\bibitem[CR15]{CR15}
I.A. Cheltsov, Y.A. Rubinstein.
\newblock {\em Asymptotically log Fano varieties.} 
\newblock Adv. Math. 285, 1241--1300 (2015).

\bibitem[CR18]{CR18}
I.A. Cheltsov, Y.A. Rubinstein.
\newblock {\em On flops and canonical metrics.} 
\newblock Ann. Sc. Norm. Super. Pisa Cl. Sci. (5) 18, 283--311 (2018).

\bibitem[PCS19]{PCS19}
V. Przyjalkowski, I. Cheltsov, C. Shramov.
\newblock {\em Fano threefolds with infinite automorphism groups.} 
\newblock Izvestiya. Mathematics, vol. 83. No. 4, 860--907 (2019).

\bibitem[CDS15]{CDS15}
X.-X. Chen, S. K. Donaldson and S. Sun.
\newblock {\em K\"ahler-Einstein metrics on Fano manifolds, I-III.} 
\newblock J. Amer. Math. Soc. 28, 183--197, 199--234, 235--278 (2015).

\bibitem[Don02]{Don02}
S. Donaldson.
\newblock {\em Scalar curvature and stability of toric varieties.} 
\newblock J. Differential Geom. 62, no. 2, 289--349 (2002).

\bibitem[Don12]{Don12}
S. Donaldson.
\newblock {\em K\"ahler metrics with cone singularities along a divisor, Essays in mathematics and its applications.} 
\newblock Springer, Heidelberg, 2012, 49--79 (2012).

\bibitem[Fu11]{Fu11}
K. Fujita.
\newblock {\em Classification list of log smooth log Fano $3$-folds.} 
\newblock Unpublished note.

\bibitem[Fu14a]{Fu14a}
K. Fujita.
\newblock {\em The Mukai conjecture for log Fano manifolds.
} 
\newblock Central European Journal of Mathematics, vol. 12, 14--27 (2014).

\bibitem[Fu14b]{Fu14b}
K. Fujita.
\newblock {\em Simple normal crossing Fano varieties and log Fano manifolds.
} 
\newblock Nagoya Math. J., 214, 95--123 (2014).

\bibitem[Fu16a]{Fu16a}
K. Fujita.
\newblock {\em On K-stability and the volume functions of Q-Fano varieties.} 
\newblock Proc. Lond. Math. Soc. 113, 541--582 (2016).

\bibitem[Fu16b]{Fu16b}
K. Fujita.
\newblock {\em On log K-stability for asymptotically log Fano varieties.} 
\newblock Annales de la Faculte des sciences de Toulouse: Mathematiques, 6, (25) no. 5, 1013--1024 (2016).

\bibitem[Fu19]{Fu19}
K. Fujita.
\newblock {\em A valuative criterion for uniform K-stability of Q-Fano varieties.} 
\newblock J. Reine Angew. Math. 751, 309--338 (2019).

\bibitem[Fu20a]{Fu20a}
K. Fujita.
\newblock {\em On K-Polystability for Log Del Pezzo Pairs of Maeda Type.} 
\newblock Acta Mathematica Vietnamica volume 45, pages 943--965 (2020).

\bibitem[Fu20b]{Fu20b}
K. Fujita.
\newblock {\em K-stability of log Fano hyperplane arrangements.} 
\newblock  arXiv:1709.08213, 2020.


\bibitem[IP99]{IP99}
V. Iskovskikh, Yu. Prokhorov.
\newblock {\em Fano Varieties, Algebraic Geometry V.}
\newblock Encyclopaedia of Mathematical Sciences, Vol. 47, 1999.

\bibitem[KM98]{KM98}
J\'a. Koll\'ar, Sh. Mori.
\newblock {\em Birational geometry of algebraic varieties.}
\newblock Cambridge tracts in mathematics (1998).

\bibitem[Laz04]{Laz04}
R. Lazarsfeld.
\newblock {\em Positivity in algebraic geometry, I: Classical setting: line bundles and linear series.}
\newblock Ergebnisse der Mathematik und ihrer Grenzgebiete. (3) 48, Springer, Berlin, 2004.

\bibitem[Li17]{Li17}
C. Li.
\newblock {\em K-semistability is equivariant volume minimization.}
\newblock Duke Math. J. 166, no. 16, 3147--3218 (2017).

\bibitem[Lo22]{Lo22}
K. Loginov.
\newblock {\em On semistable degenerations of Fano varieties.}
\newblock European J. of Math., vol. 8, 991--1005 (2022).

\bibitem[LM22]{LM22}
K. Loginov, J. Moraga.
\newblock {\em Maximal log Fano manifolds are generalized Bott towers.}
\newblock J. of Algebra, vol. 612, 110--146 (2022).

\bibitem[M83]{M83}
H. Maeda.
\newblock {\em Classification of logarithmic Fano $3$-folds.}
\newblock Proc. Japan Acad. Ser. Math. Sci. 59 (6), 245--247 (1983).

\bibitem[OS15]{OS15}
Yu. Odaka, Yuji, S. Sun.
\newblock {\em Testing Log K-stability by blowing up formalism.}
\newblock Ann. Fac. Sci. Toulouse Math, 6, Tome 24 (3), 505--522 (2015).

\bibitem[R14]{R14}
Ya. Rubinstein.
\newblock {\em Smooth and Singular K\"ahler-Einstein Metrics.}
\newblock Contemporary Mathematics, vol. 630 (2014).

\bibitem[Ta09]{Ta09}
K. Takeuchi.
\newblock {\em Weak Fano threefolds with del Pezzo fibration.}
\newblock arXiv:0910.2188, 2009.

\bibitem[Tia97]{Tia97}
G. Tian.
\newblock {\em K\"ahler-Einstein metrics with positive scalar curvature.}
\newblock Invent. Math. 130, no. 1, 1--37 (1997).

\bibitem[Tia15]{Tia15}
G. Tian.
\newblock {\em K-Stability and K\"ahler-Einstein Metrics.}
\newblock Comm. Pure Appl. Math. 68 (7), 1085--1156 (2015).

\bibitem[Xu21]{Xu21}
C. Xu
\newblock {\em K-stability of Fano varieties: an algebro-geometric approach.}
\newblock EMS Surveys in Mathematical Sciences, vol. 8 1/2, 265--354 (2021).

\bibitem[Zha06]{Zha06}
Qi Zhang.
\newblock {\em Rational connectedness of log ${\bf Q}$-Fano varieties.}
\newblock J. Reine Angew. Math., 590, 131--142 (2006).

\bibitem[Zhu20a]{Zhu20a}
Z. Zhuang.
\newblock {\em Optimal destabilizing centers and equivariant K-stability.}
\newblock arXiv:2004.09413, 2020.

\bibitem[Zhu20b]{Zhu20b}
Z. Zhuang.
\newblock {\em Product theorem for K-stability.}
\newblock Adv. Math. 371, 18 (2020).
\end{thebibliography}
\end{document}